\documentclass[12pt]{article}
\title{\textbf{ \Large Connectedness of Lakshmibai-Seshadri path crystals for hyperbolic Kac-Moody algebras of rank 2}}
\author{Ryuta Hiasa\\
	{\small Graduate School of Pure and Applied Sciences, University of Tsukuba,}\\
	{\small 1-1-1 Tennodai, Tsukuba, Ibaraki 305-8571, Japan}\\
	{\small (e-mail: \texttt{hiasa@math.tsukuba.ac.jp})}
}
\date{}

\usepackage{amsmath}
\usepackage{amssymb}
\usepackage{amsthm}
\usepackage{comment}

\usepackage{color}

\def \Z {{\mathbb Z}}

\def \R {{\mathbb R}}
\def \C {{\mathbb C}}
\def \B {{\mathbb B}}

\def \e {{\tilde{e}}}
\def \f {{\tilde{f}}}
\def \II {{I\hspace{-.1em}I}}
\def \derep {{\Delta_{\mathrm{re}}^+}}

\newtheorem{theorem}{Theorem}[section]
\newtheorem{proposition}[theorem]{Proposition}
\newtheorem{lemma}[theorem]{Lemma}
\newtheorem{corollary}[theorem]{Corollary}

\theoremstyle{definition}
\newtheorem{definition}[theorem]{Definition}

\newtheorem{remark}[theorem]{Remark}

\numberwithin{equation}{section}

\usepackage[dvipdfm, top=28truemm, bottom=28truemm, left=28truemm, right=28truemm]{geometry}
\usepackage{graphicx}

\usepackage[labelsep=period]{caption}

\begin{document}
\maketitle
\begin{abstract}
	Let $\mathfrak{g}$ be a hyperbolic Kac-Moody algebra of rank $2$.
	We give a necessary and sufficient condition for
	the crystal graph of the Lakshmibai-Seshadri path crystal $\B(\lambda)$
	to be connected
	for an arbitrary integral weight $\lambda$.
\end{abstract}

\begin{center}
{\footnotesize{keywords: Crystal, hyperbolic Kac–Moody algebras,
Lakshmibai–Seshadri paths

}}
\end{center}

\setlength {\baselineskip}{16pt}
\section{Introduction.}
Let $\mathfrak{g}$ be a symmetrizable Kac-Moody algebra over $\C$ with $\mathfrak{h}$ the Cartan subalgebra.
We denote by $W$ the Weyl group of $\mathfrak{g}$. Let $P$ be an integral weight lattice of $\mathfrak{g}$,
and $P^+$ (resp., $-P^+$) the set of dominant (resp., antidominant) integral weights in $P$.
In \cite{L2} and \cite{L},
Littelmann introduced the notion of Lakshmibai-Seshadri (LS for short) paths of shape
$\lambda \in P$, and gave the set $\B(\lambda)$ of all LS paths of shape $\lambda$ a crystal structure.
Kashiwara \cite{Ksim} and Joseph \cite{J} proved independently that
if $\lambda \in P^+$ (resp., $\lambda \in -P^+$),
then $\B(\lambda)$ is isomorphic to the
crystal basis $\mathcal{B}(\lambda)$
of the integrable highest (resp., lowest) weight module
of highest (resp., lowest) weight $\lambda$.
Also, we see by the definition of an LS path
that $\mathbb{B}(\lambda)=\mathbb{B}(w\lambda)$
for $\lambda \in P$ and $w \in W$.
Hence, if $W\lambda  \cap (P^+ \cup -P^+) \neq \emptyset$, then $\mathbb{B}(\lambda)$
is isomorphic to the crystal basis of an integrable highest (or lowest) weight module.
So, we are interested in the case that
\begin{equation}\label{int.A}
	W\lambda  \cap (P^+ \cup -P^+) = \emptyset. \tag{$*$}
\end{equation}

If $\mathfrak{g}$ is of finite type,
then $W\lambda  \cap P^+ \neq \emptyset$ for any $\lambda \in P$.
Assume that $\mathfrak{g}$ is of affine type.
Then,
$W\lambda  \cap (P^+ \cup -P^+) = \emptyset$
if and only
if ($\lambda \neq 0$, and) the level of $\lambda$ is $0$.
%
Naito and Sagaki proved in \cite{NS1} and \cite{NS2} that
if $\lambda$ is a positive integer multiple of a level-zero fundamental
weight, then $\B(\lambda)$ is isomorphic to the
crystal basis $\mathcal{B}(\lambda)$ of the
extremal weight module $V(\lambda)$ introduced in \cite{K}.
After that,
Ishii, Naito, and Sagaki \cite{INS}
introduced semi-infinite LS paths of shape $\lambda$,
for a level-zero dominant integrable weight $\lambda$,
and proved that
the crystal $\B^{\frac{\infty}{2}}(\lambda)$ of
semi-infinite LS paths of shape $\lambda$ is isomorphic to
the crystal basis $\mathcal{B}(\lambda)$ of the
extremal weight module $V(\lambda)$.
For the case that $\mathfrak{g}$ is of  indefinite type,
there are few studies on the crystal structure of
$\B(\lambda)$ (for $\lambda \in P$ satisfying the condition \eqref{int.A})
and its relationship to the representation theory.
Yu \cite{Yu} and Sagaki-Yu \cite{S-Yu} studied them in the special case that
$\mathfrak{g}=\mathfrak{g}(A)$
is the
hyperbolic Kac-Moody algebra of rank $2$ over $\C$ associated to
a generalized Cartan matrix
\begin{equation*}
	A=
	\begin{pmatrix}
		2 & -a \\-b & 2\\
	\end{pmatrix}, \quad
	\mathrm{where} \
	a, b \in \mathbb{Z}_{\geq 2} \
	\mathrm{with} \
	ab>4,
\end{equation*}
and $\lambda=\Lambda_1-\Lambda_2$,
where $\Lambda_1, \Lambda_2$ are the fundamental weights.
Namely, Yu \cite{Yu} proved that $\lambda=\Lambda_1-\Lambda_2$
satisfies the condition \eqref{int.A} and the crystal graph of $\B(\Lambda_1-\Lambda_2)$ is connected.
Then,
Sagaki and Yu \cite{S-Yu} proved that $\B(\Lambda_1-\Lambda_2)$ is isomorphic to  the crystal basis $\mathcal{B}(\Lambda_1-\Lambda_2)$ of the extremal weight module $V(\Lambda_1-\Lambda_2)$
of extremal weight $\Lambda_1-\Lambda_2$.

In this paper, as a further study after \cite{Yu},
we study the integral weights $\lambda \in P$ satisfying \eqref{int.A},
and the crystal structure of $\B(\lambda)$ (for $\lambda \in P$ satisfying \eqref{int.A})
in the case that $\mathfrak{g}=\mathfrak{g}(A)$
is the hyperbolic Kac-Moody algebra of rank $2$ as above.
First, we give a necessary and sufficient condition for
a Weyl group orbit in $P$ to satisfy the condition \eqref{int.A}.
Let $\mathbb{O}:=\{ W\lambda \mid \lambda \in P \}$
be the set of $W$-orbits in $P$.
\begin{theorem}[$=$ Theorem \ref{thm.M}]\label{int.thm1}
	A Weyl group orbit $O\in \mathbb{O}$ satisfies the condition \eqref{int.A} if and only if
	$O$ contains an integral weight of the form either
	$(\mathrm{i})$ or $(\mathrm{ii})$$\mathrm{:}$
	\begin{itemize}
		\item[$(\mathrm{i})$] $k\Lambda_1-l\Lambda_2$ for some
		$k, l \in \Z_{>0}$ such that $l\leq k<(a-1)l $$\mathrm{;}$
		\item[$(\mathrm{ii})$] $k\Lambda_1-l\Lambda_2$ for some
		$k, l \in \Z_{>0}$ such that $k<l\leq (b-1)k $.
	\end{itemize}
\end{theorem}
Next, we study the connectedness of the crystal graph of $\B(\lambda)$
for $\lambda \in P$ satisfying the condition \eqref{int.A},
which will play an important role in the future study on the relationship between
$\B(\lambda)$ and the crystal basis of an integrable module,
such as an extremal weight module.

\begin{theorem}[$=$ Theorem \ref{cor.main}]\label{int.thm2}
	Assume that  $\lambda \in P $  satisfies the condition \eqref{int.A}.
	The crystal graph of $\B(\lambda)$ is connected if and only if
	$W\lambda$
	contains an integral weight $\lambda'$ of the form$\mathrm{:}$
	$\lambda'=k'\Lambda_1-l'\Lambda_2$ with
	$|k'|=1$ or $|l'|=1$.
	Otherwise, the crystal graph of  $\B(\lambda)$ has infinitely many connected
	components.
\end{theorem}

This paper is organized as follows.
In Section 2, we fix our notation, and recall
the definitions and several properties of LS paths.
In Section 3, we give a proof of Theorem \ref{int.thm1}.
In Section 4, we give a proof of Theorem \ref{int.thm2} by dividing it into three parts.

\section{Preliminaries.}

\subsection{Hyperbolic Kac-Moody algebra of rank $2$.}
Let $A$ be a hyperbolic generalized Cartan matrix of the form
\begin{equation}\label{eq.gcm}
	A=
	\begin{pmatrix}
		2 & -a \\-b & 2\\
	\end{pmatrix}, \quad
	\mathrm{where} \
	a, b \in \mathbb{Z}_{\geq 2} \
	\mathrm{with} \
	ab>4.
\end{equation}
Let $\mathfrak{g}= \mathfrak{g}(A)$ be the Kac-Moody algebra
associated to $A$ over $\C$.
We denote by $\mathfrak{h}$ the Cartan subalgebra
of $\mathfrak{g}$,
$\{ \alpha_{i} \}_{i \in I} \subset \mathfrak{h^*}$
the set of simple roots,
and $\{ \alpha_i^\vee \}_{i \in I} \subset \mathfrak{h}$
the set of simple coroots,
where $I=\{ 1, 2\}$.
Let  $r_i$ be the simple reflection with respect to  $\alpha_i$
for $i \in I$,  and let
$W=\langle r_i \mid i \in I \rangle$
be the Weyl group of $\mathfrak{g}$.
Note that $W=\{ x_m \mid m \in \Z \}$, where
\begin{equation}\label{xm}
x_{m}:=
	\begin{cases}
		(r_2 r_1)^k & \mathrm{if} \ m=2k \ \mathrm{with} \ k \in \Z_{\ge0}, \\
		r_1(r_2 r_1)^k & \mathrm{if} \ m=2k+1 \ \mathrm{with} \ k \in \Z_{\ge0}, \\
		(r_1 r_2)^{-k} & \mathrm{if} \ m=2k \ \mathrm{with} \ k \in \Z_{\le0}, \\
		r_2(r_1 r_2)^{-k} & \mathrm{if} \ m=2k-1 \ \mathrm{with} \ k \in \Z_{\le0}. \\
	\end{cases}
\end{equation}
Let $\Delta_{\mathrm{re}}^+$ denote the set of positive real roots.
For a positive real root $\beta \in \Delta_{\mathrm{re}}^+$,
we denote by $\beta^{\vee}$ the dual root of $\beta$,
and by $r_\beta \in W $ the reflection with respect to  $\beta$.
Let
$\{ \Lambda_i \}_{i \in I} \subset \mathfrak{h^*}$ be the fundamental weights for $\mathfrak{g}$,
i.e., $\langle \Lambda_i , \alpha_j^\vee \rangle=\delta_{i, j}$ for $i, j \in I$,
{
where $\langle \cdot , \cdot \rangle : \mathfrak{h^*} \times \mathfrak{h
}\rightarrow \C$
is the canonical pairing of $\mathfrak{h^*}$ and $\mathfrak{h}$}.
We set $P:=\Z \Lambda_1 \oplus \Z \Lambda _2$;
{note that $\alpha_1=2\Lambda_1-b\Lambda_2$ and $\alpha_2=-a\Lambda_1+2\Lambda_2$}.
Let
$P^+=\Z_{\geq 0} \Lambda_1 +\Z_{\geq 0} \Lambda _2 \subset P$
be the set of dominant integral weights,
and $-P^+$ the set of antidominant integral weights.

\subsection{Lakshmibai-Seshadri paths.}\label{section.LSpath}
Let us recall the definition of a Lakshmibai-Seshadri path
from \cite[\S 2, \S 4]{L}.
In this subsection,
we fix an integral weight  $\lambda \in P$.

\begin{definition}\label{order}
For $\mu, \nu \in W\lambda$,
we write $\mu \geq \nu $
if there exist a sequence
$\mu=\mu_0, \mu_1, \ldots$, $\mu_u=\nu$
of elements in
$W\lambda$
and a sequence $\beta_1, \beta_2, \ldots , \beta_u$
of positive real roots such that
$\mu_k=r_{\beta_k}(\mu_{k-1})$ and
$\langle \mu_{k-1}, \beta^{\vee}_k \rangle < 0$
for each $k=1,2, \ldots, u$.
If $\mu \geq \nu$, then we define $\mathrm{dist}(\mu, \nu)$
to be the maximal length $u$ of all possible such sequences
$\mu=\mu_0, \mu_1, \ldots, \mu_u=\nu$.
\end{definition}

\begin{remark}\label{rem.1}
For $\mu, \nu \in W\lambda$ such that
$\mu > \nu$ and $\mathrm{dist}(\mu, \nu)=1$,
there exists a unique positive real root $\beta \in \derep$ such that
$\nu=r_{\beta}(\mu)$.
\end{remark}

{The Hasse diagram of $W\lambda$ is, by definition, the $\derep$-labeled,
directed graph with vertex set $W\lambda$,
and edges of the following form:
$\mu \xleftarrow{\beta} \nu$
for $\mu, \nu \in W\lambda$ and
$\beta \in \derep$
such that $\mu >\nu $ with $\mathrm{dist}(\mu, \nu)=1$ and $\nu=r_{\beta}(\mu)$.}

\begin{definition}
Let $\mu, \nu \in W\lambda$ with $\mu > \nu $,
and let $0<\sigma<1$ be a rational number.
A $\sigma$-chain for $(\mu, \nu)$ is a sequence
$\mu = \mu_0, \ldots, \mu_u=\nu$
of elements of $W\lambda $ such that
$\mathrm{dist}(\mu_{k-1}, \mu_k)=1$
and $ \sigma \langle \mu_{k-1},
\beta_k^\vee \rangle \in \mathbb{Z}_{<0}$
for all $k=1,2,\ldots, u$,
where $\beta_k$ is the unique positive real root satisfying
$\mu_k=r_{\beta_k}(\mu_{k-1})$.
\end{definition}

\begin{definition}
Let $\mu_1> \cdots> \mu_u$ be a finite sequence of elements in $W\lambda$,
and let $0=\sigma_0<\cdots<\sigma_u=1$ be a finite sequence of rational numbers.
The pair
$\pi=(\mu_1,\ldots,\mu_u; \sigma_0, \ldots, \sigma_u)$
is called a Lakshmibai-Seshadri (LS for short) path of shape $\lambda$
if there exists
a $\sigma_k$-chain for $(\mu_k,\mu_{k+1})$
for each $k=1, \ldots, u-1$.
We denote by $\mathbb{B}(\lambda)$
the set of LS paths of shape $\lambda$.
\end{definition}

We identify
$\pi=(\mu_1,\ldots,\mu_u;\sigma_0, \ldots , \sigma_u) \in \mathbb{B}(\lambda)$ with the following piecewise-linear continuous map
$\pi:[0,1]\to \R \otimes_{\Z} P $, where $[0,1]:= \{ t\in \R \mid 0 \leq t \leq 1 \}$:
\begin{equation*}
	\pi(t)=\sum^{j-1}_{k=1}(\sigma_k-\sigma_{k-1})\mu_k+
	(t-\sigma_{j-1})\mu_j  \ \text{for} \
	\sigma_{j-1}\le t\le\sigma_j, \  1\le j\le u.
\end{equation*}

Now, we endow $\mathbb{B}(\lambda)$ with
a crystal structure as follows;
for the definition of a crystal, see, e.g., \cite{HK}.
First, we define $\mathrm{wt}(\pi):=\pi(1)$
for $\pi \in \mathbb{B}(\lambda) $;
we know from \cite[Lemma 4.5 (a)]{L} that $\pi(1)\in P$.
Next, for $\pi \in \mathbb{B}(\lambda) $ and $i \in I$,
we define
\begin{equation}
	H^\pi_i(t):=\langle\pi(t),\alpha_i^{\vee}\rangle \
	\text{for} \  0 \leq t \leq1,
\end{equation}
\begin{equation}\label{eq.min}
	m^\pi_i:=\mathrm{min}\{H^\pi_i(t) \mid 0\leq t \leq 1\}.
\end{equation}
From \cite[Lemma 4.5 (d)]{L}, we know that
\begin{equation}\label{int}
	\text{all local minimal values of} \ H^\pi_i(t) \
	\text{are integers};
\end{equation}
in particular, $m_i^\pi \in \mathbb{Z}_{\leq0}$
and $H^\pi_{i}(1)-m_i^\pi \in \mathbb{Z}_{\geq0}$.
We define $\e_i \pi$ as follows:
If $m^\pi_i=0$, then we set $\e_i\pi:=\mathbf{0}$,
where $\mathbf{0}$ is an extra element not contained in any crystal.
If $m^\pi_i\le-1$, then we set
\begin{align}
t_1&:=\mathrm{min}\{t\in[0,1] \mid H^\pi_i(t)=m^\pi_i\}, \label{et_1} \\
t_0&:=\mathrm{max}\{t\in[0,t_1] \mid H^\pi_i(t)=m^\pi_i+1\}; \label{et_0}
\end{align}
we see by (\ref{int}) that
$H^\pi_i(t)$ is strictly decreasing on $[t_0, t_1]$.
We define
\begin{equation}
	(\e_i\pi)(t):=
	\begin{cases}
		\pi(t)
			& \text{if} \ 0\le t\le t_0, \\
		r_i(\pi(t)-\pi(t_0))+\pi(t_0)
			& \text{if} \ t_0\le t\le t_1,\\
		\pi(t)+\alpha_i
			& \text{if} \ t_1\le t \leq 1;
	\end{cases}
\end{equation}
we know from
\cite[$\S 4$]{L} that $\e_i \pi \in \mathbb{B}(\lambda)$.
Similarly, we define
$ \f_i \pi$ as follows:
If $H^\pi_i(1)-m_i^\pi=0$, then we set $\f_i\pi:=\mathbf{0}$.
If $H^\pi_i(1)-m_i^\pi \geq 1$, then  we set
\begin{align}
	t_0 &:=\mathrm{max}\{t\in[0,1]
		\mid 	H^\pi_i(t)=m^\pi_i\}, \label{ft_0} \\
	t_1&:=\mathrm{min}\{t\in[t_0,1] \mid 		H^\pi_i(t)=m^\pi_i+1\}; \label{ft_1}
\end{align}
we see {by} (\ref{int}) that
$H^\pi_i(t)$ is strictly increasing on $[t_0, t_1]$.
We define
\begin{equation}
	(\f_i\pi)(t):=
	\begin{cases}
		\pi(t)
			& \text{if }  0\leq t\leq t_0 , \\
		r_i(\pi(t)-\pi(t_0))+\pi(t_0)
			& \text{if }  t_0 \leq t\leq t_1, \\
		\pi(t)-\alpha_i
			& \text{if } t_1\leq t \leq1;
	\end{cases}
\end{equation}
we know from \cite[$\S 4$]{L}
that $f_i \pi \in \mathbb{B}(\lambda)$.
We set $\e_i\mathbf{0}= \f_i\mathbf{0} := \mathbf{0}$ for
$i \in I $.
%
Finally,
for $\pi \in \mathbb{B}(\lambda)$ and $i \in I$,
we set
\begin{equation*}
	\varepsilon_i (\pi)
	:=\mathrm{max}\{k\in\mathbb{Z}_{\geq 0}
	\mid \e^k_i \pi \neq\mathbf{0}\}, \quad
	\varphi_i (\pi) :=\mathrm{max}\{k\in\mathbb{Z}_{\geq 0}
	\mid \f^k_i \pi \neq\mathbf{0}\}.
\end{equation*}
We know from {\cite[Lemma 2.1 (c)]{L}} that
$\varepsilon_i(\pi)=-m^\pi_i$ and
$\varphi_i (\pi)=H^\pi_i(1)-m_i^\pi$.
\begin{theorem}[{\cite[\S 2, \S 4]{L}}]
The set $\mathbb{B}(\lambda)$,
together with the maps
$\mathrm{wt}: \mathbb{B}(\lambda)  \to P$,
$\e_i, \f_i: \mathbb{B}(\lambda) \to
	\mathbb{B}(\lambda) \cup \{ \mathbf{0} \}$, $i\in I$,
and $\varepsilon_i, \varphi_i : \mathbb{B}(\lambda) \to
	 \Z_{\geq0}$, $i \in I$, {is} a crystal.
\end{theorem}

For $\pi=(\mu_1,\ldots,\mu_u; \sigma_0, \ldots, \sigma_u)
\in \mathbb{B}(\lambda)$,
we set $\iota(\pi):=\mu_1$ and $\kappa(\pi):=\mu_u$.
For $\pi \in\mathbb{B(\lambda)}$ and $i \in I$,
we set
$\e_i^{\mathrm{max}}\pi:= \e_i^{\varepsilon_i(\pi)}\pi$ and
$\f_i^{\mathrm{max}}\pi:= \f_i^{\varphi_i(\pi)} \pi$.

\begin{lemma}[{\cite[Proposition 4.7]{L}}]\label{IK}
Let $\pi \in\mathbb{B(\lambda)}$, and $i \in I$.
If $\langle \iota(\pi), \alpha_i^{\vee} \rangle <0$,
then $\iota(e_i^{\mathrm{max}}\pi)=r_i\iota(\pi)$.
If $\langle \kappa(\pi), \alpha_i^{\vee} \rangle >0$,
then $\kappa(f_i^{\mathrm{max}}\pi)=r_i\kappa(\pi)$.
\end{lemma}

\section{Weyl group orbits.}
\subsection{Weyl group orbit $O $ satisfying
$O \cap (P^+ \cup -P^+) = \emptyset$.}
If $\lambda \in P $ is a dominant (resp., antidominant)
weight, then $\B(\lambda)$ is isomorphic to the
crystal basis $\mathcal{B}(\lambda)$
of the integrable highest (resp., lowest) weight module
of highest (resp., lowest) weight $\lambda$ (see \cite[Theorem 4.1]{Ksim}, \cite[\S 6.4.27]{J}).
Also, we see by the definition of  LS paths
that $\mathbb{B}(\lambda)=\mathbb{B}(w\lambda)$
for $\lambda \in P$ and $w \in W$.
So, we focus on those $\lambda \in P $
satisfying the condition that
\begin{equation}\label{A}
	W\lambda  \cap (P^+ \cup -P^+) = \emptyset.
\end{equation}
Let $\mathbb{O}:=\{ W\lambda \mid \lambda \in P \}$ be
the set of $W$-orbits in $P$.
Recall that the {generalized} Cartan matrix of $\mathfrak{g}$ is of the form \eqref{eq.gcm}.
\begin{theorem}\label{thm.M}
	A $W$-orbit $O\in \mathbb{O}$ satisfies the condition \eqref{A},
	that is, $O  \cap (P^+ \cup -P^+) = \emptyset$,  if and only if
	$O$ contains $\lambda \in P$ of the form either
	$(\mathrm{i})$ or $(\mathrm{ii})$$\mathrm{:}$
	\begin{itemize}
		\item[$(\mathrm{i})$] $k\Lambda_1-l\Lambda_2$ for some
		$k, l \in \Z_{>0}$ such that $l\leq k<(a-1)l $$\mathrm{;}$
		\item[$(\mathrm{ii})$] $k\Lambda_1-l\Lambda_2$ for some
		$k, l \in \Z_{>0}$ such that $k<l\leq (b-1)k $.
	\end{itemize}
\end{theorem}
\begin{remark}\label{rem.main1}
	We claim that
	$O\in \mathbb{O}$ contains $\lambda \in P$ of the form
	either  $(\mathrm{i})$ or $(\mathrm{ii})$ in Theorem \ref{thm.M}
	if and only if
	$O$ contains $\lambda' \in P$ of the form
	either
	$(\mathrm{i'})$--$(\mathrm{iv'})$$\mathrm{:}$
	\begin{itemize}
		\item[$(\mathrm{i'})$] $k\Lambda_1-k\Lambda_2$ for some
		$k \in \Z_{>0}$$\mathrm{;}$
		\item[$(\mathrm{ii'})$] $k\Lambda_1-(b-1)k\Lambda_2$ for some
		$k \in \Z_{>0}$$\mathrm{;}$
		\item[$(\mathrm{iii'})$] $k\Lambda_1-l\Lambda_2$ for some
		$k, l \in \Z_{>0}$ such that $l<k<(a-1)l $$\mathrm{;}$
		\item[$(\mathrm{iv'})$] $k\Lambda_1-l\Lambda_2$ for some
		$k, l \in \Z_{>0}$ such that $k<l< (b-1)k $.
	\end{itemize}
	Indeed, the ``only if'' part is obvious.
	We show the ``if'' part.
	If $\lambda'$ is of the form  (iii') (resp., (iv')),
	then it is obvious that $\lambda'$ is of the form  $(\mathrm{i})$ (resp., $(\mathrm{ii})$).
	Assume that $\lambda'$ is of the form (i').
	If $a\geq 3$, then $\lambda'$ is of the form $(\mathrm{i})$.
	If $a =2$, then we see that {$b\geq 3$ and} $O$ contains $r_1 r_2\lambda'=r_1 r_2(k\Lambda_1-k\Lambda_2)=k\Lambda_1-(b-1){k}\Lambda_2 $,
	which is of the form $(\mathrm{ii})$.
	Assume that $\lambda'$ is of the form (ii').
	If $b\geq 3$, then  $\lambda'$ is of the form $(\mathrm{ii})$.
	If $b =2$, then we see that {$a\geq3$ and} $\lambda'=k\Lambda_1-(b-1)k\Lambda_2=k\Lambda_1-k\Lambda_2$
	is of the form $(\mathrm{i})$.
\end{remark}
The rest of this subsection is devoted to a proof of Theorem \ref{thm.M}.
For $\lambda \in P$ of the form:
$\lambda=k\Lambda_1-l\Lambda_2$ with  $k, l \in \Z$,
we define the sequence $\{ p_m \}_{m\in \Z} $ of integers
by the following recursive formulas: For $m\geq 0$,
\begin{equation}\label{eq.pm}
	p_0=l, \quad
	p_1=k, \quad
	p_{m+2}=
	\begin{cases}
		bp_{m+1}-p_m & \text{if} \ m \ \text{is even}, \\
		ap_{m+1}-p_m & \text{if} \ m \ \text{is odd}; \\
	\end{cases}
	\end{equation}
for  $m<0$,
	\begin{equation}\label{eq.pm2}
		p_{m}=
		\begin{cases}
			bp_{m+1}-p_{m+2} & \text{if} \ m \ \text{is even}, \\
			ap_{m+1}-p_{m+2} & \text{if} \ m \ \text{is odd}. \\
		\end{cases}
\end{equation}
By induction on $|m|$, we can show the following lemma.
\begin{lemma}\label{lem.xm}
	For $m \in \Z$,
	\begin{equation*}
		x_m\lambda=
		\begin{cases}
			p_{m+1}\Lambda_1-p_m\Lambda_2 &
			\text{if}\ m \ \text{is even},\\
			-p_m\Lambda_1+p_{m+1}\Lambda_2 &
			\text{if}\ m \ \text{is odd}.\\
		\end{cases}
	\end{equation*}
\end{lemma}
\begin{corollary}\label{cor.pos}
	Let $\lambda=k\Lambda_1-l\Lambda_2 \in P$ be
	an integral weight.
	The Weyl group orbit $W\lambda\in \mathbb{O}$
	satisfies the condition \eqref {A}
	if and only if $p_m>0$ for all $m \in \Z$
	or $p_m<0$  for all $m \in \Z$.
\end{corollary}
\begin{lemma}\label{lem.p}
	Let $\lambda=k\Lambda_1-l\Lambda_2\in P$.\\
$(1)$
		If there exists $n' \in \Z_{\geq 0}$
		such that $0<p_{n'} <p_{n'+1} $,
		then $0 < p_{n} < p_{n+1}$ for all $n\geq n'$.\\
$(2)$
		If there exists $n' \in \Z_{\leq 1}$
		such that $0<p_{n'} <p_{n'-1} $,
		then $ 0 < p_{n} < p_{n-1} $ for all  $n \leq n'$.
\end{lemma}

\begin{proof}
	We give a proof only for part (1);
	the proof for part (2) is similar.
	We {proceed} by induction on $n$.
	The assertion is trivial when $n=n'$.
	Assume that $n>n'$.
	We set
	\begin{equation*}
		a':=
		\begin{cases}
			a & \text{if } n \text{ is even},\\
			b & \text{if } n \text{ is odd};
		\end{cases}
	\end{equation*}
	note that $p_{n+1}=a'p_{n}-p_{n-1}$.
	Then we compute
	\begin{equation*}
		p_{n+1}-p_{n}=(a'p_{n}-p_{n-1})-p_{n}=(a'-1)(p_{n}-p_{n-1})+(a'-2)p_{n-1}.
	\end{equation*}
	Because  $a'\geq 2$ (see \eqref{eq.gcm}),
	and {$p_n>p_{n-1}>0$} by the
	 induction hypothesis, we obtain $p_{n+1}>p_n$ as desired.
\end{proof}


\begin{proposition}\label{prop.orbit}
	Let $\lambda=k\Lambda_1-l\Lambda_2 \in P$
	with $k, l >0$.
	If $p_m \neq p_{m+1}$ for any $m \in \Z $,
	then the following are equivalent.
	\begin{enumerate}
		\item[$(1)$]
			The Weyl group orbit
			$W\lambda$ satisfies the condition \eqref{A},
			or equivalently, $p_m>0$
			for all $m \in \Z $ by $\mathrm{Corollary}$ $\ref{cor.pos}$
			and the assumption that $k, l >0$.
		\item[$(2)$]
			There exists an element
			$\lambda'=k'\Lambda_1-l'\Lambda_2$ in  $W\lambda$ satisfying
			the {conditions} that
			$k', l' \in \Z_{>0}$,
			and  $l'<k'<(a-1)l'$ or $k'<l'<(b-1)k'$.
	\end{enumerate}
\end{proposition}
\begin{proof}
	$(1)\Rightarrow (2)$:
	Since $p_m$ is a positive integer for every $m \in \Z$ by the assumption in (1),
	and since $p_m \neq p_{m+1}$ for any $m \in \Z $ by the assumption,
	there exists $n\in\Z $ such that
	$p_{n-1}> p_{n}< p_{n+1}.$
	If $n$ is even, then we have
	$(a-1)p_{n}-p_{n+1} = p_{n-1}-p_{n}>0 $
	by {\eqref{eq.pm} and \eqref{eq.pm2}}.
	Hence, $\lambda' := x_{n}\lambda =p_{n+1}\Lambda_1-{p_n}\Lambda_2 $
	satisfies the condition $p_n<p_{n+1}<(a-1)p_{n}$.
	Similarly, if $n$ is odd, then we have
	$(b-1)p_{n}-p_{n-1} = p_{n+1}-p_{n}>0 $ by {\eqref{eq.pm} and \eqref{eq.pm2}}.
	Hence, $\lambda' := x_{n-1}\lambda=p_{n}\Lambda_1- p_{n-1}\Lambda_2 $
	satisfies the condition $p_{n}<p_{n-1}<(b-1)p_n$.

	$(2)\Rightarrow (1)$:
	Assume that $W\lambda $ contains an element $\lambda'$ of the form:
	 $\lambda'=k'\Lambda_1-l'\Lambda_2 \in W\lambda$ with
	$k', l' \in \Z_{>0}$ such that  $l'<k'<(a-1)l'$
	(resp., $k'<l'<(b-1)k'$).
	We define the sequence $\{ p'_m\}_{m\in \Z}$ for $\lambda'$
	in the same manner as  (\ref{eq.pm}) and \eqref{eq.pm2}:
	\begin{equation*}
		p'_0=l',  \quad
		p'_1=k',  \quad
		p'_{m+2}=
		\begin{cases}
			bp'_{m+1}-p'_m & \text{if} \ m \ \text{is even}, \\
			ap'_{m+1}-p'_m & \text{if} \ m \ \text{is odd}. \\
		\end{cases}
	\end{equation*}
	Since $l'<k'<(a-1)l'$ (resp., $k'<l'<(b-1)k'$),
	it is easy to check that
	$p'_{-1}>p'_0<p'_1$ (resp.,  $p'_0>p'_1< p'_{2}$).
	Applying Lemma \ref{lem.p},
	we obtain $p'_m>0$ for all $m \in \Z$.
	Hence,
	we see from Corollary \ref{cor.pos} that $W\lambda'=W\lambda$ satisfies the condition \eqref{A}.
	Thus, we have proved the proposition.
\end{proof}

\begin{remark}\label{rem.orbit}
	By Lemma \ref{lem.p} and the proof of Proposition \ref{prop.orbit},
	we see that if $\lambda$ is of the form $(\mathrm{iii'})$ (resp., $(\mathrm{iv'})$) in Remark \ref{rem.main1},
	then
	\begin{align*}
			&\cdots> p_{-1}> p_0=l <p_1=k< p_2< \cdots\\
		&\text{(resp., $\cdots> p_{-1}> p_0=l  >p_1=k< p_2< \cdots$),}
	\end{align*}
	where the sequence $\{p_m\}_{m\in\Z}$ is defined by the recursive formulas \eqref{eq.pm} and \eqref{eq.pm2} for $\lambda$.
\end{remark}

\begin{proof}[Proof of Theorem \ref{thm.M}]
	By Remark \ref{rem.main1}, it suffices to show that $O$ satisfies the condition \eqref{A}
	if and only if $O$ contains $\lambda \in P $ of the form either (i')--(iv') in Remark \ref{rem.main1}.

	First, we prove the ``if'' part.
	We know from \cite[Proposition 3.1.1]{Yu} that
	if $\mu=\Lambda_1-\Lambda_2$, then  $W\mu$
	satisfies the condition \eqref{A}.
	Hence, $W(k\mu)$ also satisfies the condition \eqref{A}
	for every $ k \in \Z \backslash  \{ 0 \}$.
	Since $r_1(k\Lambda_1-(b-1)k\Lambda_2)=-k\Lambda_1+k\Lambda_2={-k\mu}$,
	we see that for $\lambda$ of the form $(\mathrm{ii'})$,
	$W\lambda$ satisfies the condition \eqref{A}.
	Also,
	we see from  (2) $\Rightarrow$  (1) in Proposition \ref{prop.orbit} that for $\lambda$ of the form $(\mathrm{iii'})$ or $(\mathrm{iv'})$,
	$W\lambda$ satisfies the condition \eqref{A}.
	Thus we have proved the ``if'' part.

	Next, we prove the ``only if'' part.
	Assume that $O\in\mathbb{O}$ satisfies the condition \eqref{A}.
	By Lemma \ref{lem.xm}, we see that $O$ contains $\lambda=k\Lambda_1-l\Lambda_2$
	such that $k, l >0 $.
	Then we define the
	sequence $\{p_m\}_{m\in \Z}$ by the recursive formulas
	\eqref{eq.pm} and \eqref{eq.pm2} for this $\lambda$.
	If $p_m=p_{m+1}$ for some $m \in \Z$,
	we see by Lemma \ref{lem.xm} that $O=W\lambda$ contains $p_m\Lambda_1-p_m\Lambda_2$ or
	$-p_m\Lambda_1+p_m\Lambda_2=r_1(p_m\Lambda_1-(b-1)p_m\Lambda_2)$.
	Hence,
	$W\lambda$ contains an integral weight
	of the form either  $(\mathrm{i'})$ or ($\mathrm{ii'}$).
	If $p_m\neq p_{m+1}$ for any $m\in\Z$,
	then we see from (1) $\Rightarrow$ (2) in Proposition \ref{prop.orbit}
	that $O=W\lambda$ contains an integral weight of the form $(\mathrm{iii'})$ or $(\mathrm{iv'})$.
	Thus we have proved Theorem \ref{thm.M}.
\end{proof}
\subsection{Hasse diagram of $W\lambda $.}
In this subsection,
we assume that $\lambda$ is of the form either (i) or (ii)
in Theorem \ref{thm.M}.
We draw the Hasse diagram of $W\lambda$
(in the ordering of Definition \ref{order}).
Recall that $p_m>0$ for all $m\in \Z$ (by Corollary \ref{cor.pos}).

\begin{proposition}[{cf. \cite[Proposition 3.2.5]{Yu}}]\label{prop.Hasse}
	The Hasse diagram of $W\lambda$ is
	\begin{equation}\label{eq.Hasse}
		\cdots \xleftarrow{\alpha_1} x_2\lambda \xleftarrow{\alpha_2} x_1\lambda \xleftarrow{\alpha_1}
		x_0\lambda
		\xleftarrow{\alpha_2} x_{-1}\lambda \xleftarrow{\alpha_1} x_{-2}\lambda
		\xleftarrow{\alpha_2} \cdots.
	\end{equation}
\end{proposition}

\begin{proof}
	For $m \in \Z$, we set
	\begin{equation*}
		i:=
		\begin{cases}
			2 & \text{if} \ m \ \text{is even}, \  \\
			1  & \text{if} \ m \ \text{is odd}. \
		\end{cases}
	\end{equation*}
	Since $r_ix_m\lambda=x_{m-1}\lambda$ and
	$\langle x_m\lambda, \alpha_i^\vee \rangle=-p_m<0$ for every $m \in \Z$ by Lemma \ref{lem.xm},
	we have $x_m\lambda > x_{m-1}\lambda$.
	Hence, we have
	\begin{equation}\label{eq.xmcontra}
		\cdots > x_2\lambda > x_1\lambda >
		x_0\lambda
		> x_{-1}\lambda > x_{-2}\lambda
		> \cdots;
	\end{equation}
	{it is obvious from \eqref{eq.xmcontra}
	that $\mathrm{dist}(x_{m}\lambda, x_{m-1}\lambda)=1$.}
	Thus, we have proved the proposition.
\end{proof}
\section{Connectedness of the crystal graph of $\B(\lambda)$.}

\subsection{Main result.}\label{S.main}
\begin{theorem}\label{cor.main}
	Let $\lambda = k\Lambda_1-l\Lambda_2$ be an integral weight such that
	$W\lambda$ satisfies the condition \eqref{A};
	see also $\mathrm{Theorem}$ $\ref{thm.M}$.
	The crystal graph of $\B(\lambda)$ is connected if and only if
	$W\lambda$ contains $\lambda'=k'\Lambda_1-l'\Lambda_2$ with
	$k',  l' \in \Z$ satisfying $|k'|=1$ or $|l'|=1$.
	Otherwise, the crystal graph of  $\B(\lambda)$ has infinitely many connected
	components.
\end{theorem}

We will prove Theorem \ref{cor.main} by showing the following three propositions;
in these propositions,  $\lambda = k\Lambda_1-l\Lambda_2$ is an integral weight such that
$W\lambda$ satisfies the condition \eqref{A}.

\begin{proposition}[will be proved in \S \ref{S.prf.main1}]\label{thm.main1}
	The crystal graph of $\B(\lambda)$ is connected
	if either
	\eqref{eq.m1} or \eqref{eq.m2} holds$:$
	\begin{equation}\label{eq.m1}
		l=1 \text{ and } 1 \leq k< a-1;
	\end{equation}
	\begin{equation}\label{eq.m2}
		k=1 \text{ and } 1 < l\leq b-1.
	\end{equation}
\end{proposition}

\begin{proposition}[will be proved in \S \ref{S.prf.main2}]\label{thm.main2}
	The crystal graph of $\B(\lambda)$  has infinitely many connected
	components
	if $k$ and $l$ are relatively prime, and
	either $(\ref{eq.m3})$ or $(\ref{eq.m4})$ holds${:}$
	\begin{equation}\label{eq.m3}
		1<l<k< (a-1)l;
	\end{equation}
	\begin{equation}\label{eq.m4}
		1<k<l< (b-1)k.
	\end{equation}
\end{proposition}
	\begin{proposition}[will be proved in \S \ref{S.prf.main3}]\label{thm.main3}
		The crystal graph of $\B(\lambda)$
		has infinitely many connected
		components
		if $k$ and $l$ are not relatively prime, {and
		either $(\ref{eq.m3})$ or $(\ref{eq.m4})$ holds}.
	\end{proposition}
\subsection{Proof of Proposition \ref{thm.main1}.}\label{S.prf.main1}
\begin{lemma}[cf. {\cite[Lemma 4.1.1, Theorem 4.1.2]{Yu}}]\label{lem.prime}
	Assume that $\lambda=k\Lambda_1-l\Lambda_2 \in P$ is of the form
	either $\mathrm{(i)}$ or $\mathrm{(ii)}$ in $\mathrm{Theorem}$ $\ref{thm.M}$$;$
	recall that $W\lambda$ satisfies the condition \eqref{A}.
	In addition, assume that $k$ and $l$ are relatively prime.\\
	$(1)$
			For every  $m \in \Z $, the numbers $p_m$ and $p_{m+1}$
			(defined by \eqref{eq.pm} and \eqref{eq.pm2} for $\lambda$)
			are relatively prime.\\
		$(2)$
			Let $0<\sigma <1$ be a rational number,
			and let $\mu, \nu \in W\lambda $ be such that $\mu > \nu $.
			If $\mu =\mu_0> \mu_1> \cdots > \mu_s=\nu $ is
			a $\sigma$-chain for $(\mu ,\nu)$, then $s=1$.\\
		$(3)$
			An LS path $\pi $ of shape $\lambda$ is of the form
			\begin{equation}\label{eq.form}
				(x_{m+s-1}\lambda, \ldots, x_{m+1}\lambda, x_m\lambda;
				\sigma_0, \sigma_1, \ldots, \sigma_s ),
			\end{equation}
			where $m\in \Z$, $s \geq0 $, and $0=\sigma_0< \sigma_1< \cdots< \sigma_s=1 $
			satisfy the condition that $p_{m+s-v}\sigma_v \in \Z$
			for $1 \leq v \leq s-1$.
\end{lemma}
\begin{proof}
	$(1)$
	It can be easily shown by induction on $|m|$.\\
	$(2)$
	Suppose, for a contradiction, that $s \geq 2$.
	Since $\mathrm{dist}(\mu_{v-1}, \mu_{v})=1$ for every $v=1,2,\ldots, s$
	by the definition of a $\sigma$-chain, it follows from Proposition \ref{prop.Hasse}
	that there exists $m\in \Z$ such that
	$\mu_v = x_{m-v}\lambda$ for $v= 0, 1, \ldots, s$.
	We set
	\begin{equation*}
		i:=
		\begin{cases}
			2 & \text{if} \ m \ \text{is even}, \  \\
			1  & \text{if} \ m \ \text{is odd}, \
		\end{cases}
		\quad
		j:=
		\begin{cases}
			1 & \text{if} \ m \ \text{is even}, \  \\
			2  & \text{if} \ m \ \text{is odd}. \
		\end{cases}
	\end{equation*}
	By Lemma \ref{lem.xm}, we see that
	$\langle x_m\lambda ,\alpha_i^\vee \rangle=-p_m, \,
	\langle x_{m-1}\lambda ,\alpha_j^\vee \rangle =-p_{m-1}$.
	Since $-p_m$ and $-p_{m-1}$ are relatively prime by part $(1)$,
	there does not exist  $0<\sigma<1$ such that
	$\sigma \langle x_m\lambda ,\alpha_i^\vee \rangle =-\sigma p_m \in \Z_{<0}$ and
	$\sigma \langle x_{m-1}\lambda ,\alpha_j^\vee \rangle =-\sigma p_{m-1} \in \Z_{<0}$.
	This contradicts the assumption that the sequence is a $\sigma$-chain.\\
	$(3)$
	It follows immediately from the definition of an LS path and part $(2)$.
\end{proof}
\begin{remark}\label{rem.alt}
Let
	$\pi = (\nu_1, \nu_2, \ldots, \nu_s;
	\sigma_0, \sigma_1, \ldots, \sigma_s )$
	be an LS path of shape $\lambda$
	satisfying the assumptions of Lemma \ref{lem.prime}.
	By Lemma \ref{lem.xm} and Lemma \ref{lem.prime} (3), we see that
	the function $H_i^\pi(t)$ for $i \in I$ attains its minimal
	and maximal values at $t=\sigma_u$, $u=0, 1, \ldots, s$, alternately.
	Namely,
	if $H_i^\pi(t)$ for $i\in I$ attains a minimal (resp., maximal) value at
	$t=\sigma_v$, then $H_i^\pi(t)$
	attains a minimal (resp., maximal) value at $t=\sigma_u$ for all  $u=0, 1, \ldots, s$
	such that $u  \equiv v$ mod $2$.
\end{remark}

{In the remainder of this subsection, we give the proof of Proposition \ref{thm.main1}
for the case that $\lambda$ satisfies \eqref{eq.m1};
the other cases are left to the reader.
Our proof is adapted from \cite[Proof of Theorem 3.2.1]{Yu}.}

\begin{proposition}\label{LS2}
Let $\pi \in \mathbb{B}({\lambda})$,
and write it as (see $\mathrm{Lemma}$ $\ref{lem.prime}$ $(3)$)$:$
\begin{equation}\label{eq.LS1}
	\pi=(x_m{\lambda}, x_{m-1}{\lambda}{,}  \ldots ,  x_{n+1}{\lambda} , x_{n}{\lambda};
	\sigma_0, \ldots, \sigma_{m-n+1})
\end{equation}
for some $n\leq m$ and $0=\sigma_0< \cdots< \sigma_{m-n+1}=1$.
Then,  $0 \leq n \leq m$ or $n \leq m \leq -1$ holds.
\end{proposition}

\begin{proof}
	Suppose, for a contradiction, that $m\geq0 $  and $n\leq -1$.
	By the definition of an LS path,
	there exists a $\sigma_{m+1}$-chain for $({\lambda} , x_{-1}{\lambda})$.
	It follows from {Proposition} \ref{prop.Hasse}
	that $\mathrm{dist}({\lambda}, x_{-1}{\lambda})=1$ and $r_2{\lambda}=x_{-1}{\lambda}$.
	Thus, we obtain $\langle {\lambda} , \alpha_2^\vee \rangle=-1$ and $0<\sigma_{m+1}<1$,
	which contradicts $\sigma_{m+1}\langle {\lambda} , \alpha_2^\vee \rangle \in \Z$.
\end{proof}

\begin{theorem}\label{thm.LScon}
	For each $\pi \in \mathbb{B}({\lambda})$,
	$\pi=\f_{i_r} \cdots \f_{i_1}\pi_{{\lambda}}$
	or $\pi=\e_{i_r} \cdots \e_{i_1}\pi_{{\lambda}}$ for some $i_1, \ldots , i_r \in I$,
	where $\pi_{{\lambda}}:=({\lambda}; 0, 1 )$.
	In particular, the crystal graph of $\mathbb{B}({\lambda})$ is connected.
\end{theorem}

\begin{proof}
	Write $\pi \in \B({\lambda})$ as (\ref{eq.LS1}). From Proposition \ref{LS2},
	it follows that either $0 \leq n \leq m$ or $n \leq m \leq -1$.
	We show by induction on $m$ that  if $0 \leq n \leq m$,
	then $\pi=\f_{i_r} \cdots \f_{i_1}\pi_{{\lambda}}$
	for some $i_1, \ldots , i_r \in I$.
	If $m=0$, then $n=0$, and hence $\pi =\pi_{{\lambda}}$.
	Thus the claim is obvious.
	Assume that $m>0$. We set
	\begin{equation*}
		i:=
		\begin{cases}
			2 & \text{if} \ m \ \text{is even}, \  \\
			1  & \text{if} \ m \ \text{is odd}; \
		\end{cases}
	\end{equation*}
	note that $\langle x_m {\lambda} , \alpha_i^\vee \rangle<0$
	and $r_i x_m {\lambda}= x_{m-1}{\lambda}$ {by Proposition \ref{prop.Hasse}}.
	By Lemma \ref{IK}, we see that
	$\e_i^\mathrm{max}\pi \in \B({\lambda})$ satisfies
	$\iota ({\e_i}^\mathrm{max}\pi) = r_i \iota (\pi) = r_i x_m {\lambda} = x_{m-1}{\lambda}$ .
	Hence, by induction hypothesis,
	$\e_i^\mathrm{max}\pi=\f_{i_r} \cdots \f_{i_1} \pi_{{\lambda}}$
	for some $i_1, \ldots , i_r \in I$.
	Hence, we obtain $\pi =\f_i^{\varepsilon_i(\pi)} \f_{i_r} \cdots \f_{i_1}  \pi_{{\lambda}}$,
	as desired.
	Similarly,
	we can show that  if $n \leq m \leq -1$,
	then  $\pi=\e_{i_r} \cdots \e_{i_1}\pi_{{\lambda}}$
	for some $i_1, \ldots , i_r \in I$.
	Thus we have proved Theorem \ref{thm.LScon}.
\end{proof}

\subsection{Proof of Proposition \ref{thm.main2}.}\label{S.prf.main2}
{Throughout this subsection, we} assume that
$\lambda=k\Lambda_1-l\Lambda_2\in P$ is
of the form either $(\mathrm{i})$ or $(\mathrm{ii})$ in Theorem \ref{thm.M},
and $k$ and $l$ are relatively prime.
We give a proof of Proposition \ref{thm.main2} only for the case
that \eqref{eq.m3} holds, i.e., $1<l<k<(a-1)l$;
the proof for the case that \eqref{eq.m4} holds is similar.
There exists a (unique) integer
	$c \in \{1,2, \ldots, k-1 \}$ such that
	\begin{equation}\label{eq.c}
		\frac{c}{k}< \frac{1}{l}< \frac{c+1}{k}.
	\end{equation}
%
Then we define the sequence $\{q_m\}_{m\in \Z }$ of integers by the following recursive formula:
\begin{equation}\label{eq.qm}
	q_0=1, \quad
	q_1=c, \quad
	q_{m+2}=
	\begin{cases}
		bq_{m+1}-q_m & \text{if} \ m \ \text{is even}, \\
		aq_{m+1}-q_m & \text{if} \ m \ \text{is odd}. \\
	\end{cases}
\end{equation}
Recall that  $\{p_m\}_{m\in \Z}$ is defined by the recursive formulas \eqref{eq.pm} and \eqref{eq.pm2}
(for $\lambda=k\Lambda_1-l\Lambda_2$).
\begin{lemma}\label{lem.pq}
	It hold that $0<q_m < p_m $ and $q_m p_{m+1}- q_{m+1} p_{m} =k-lc$ for all $m \in \Z$.
	{In particular, we have $0<q_{m+1}/p_{m+1}<q_{m}/p_{m}<1$ for all $m\in \Z$.}
\end{lemma}
\begin{proof}
	First, let us show that $q_m >0$ for all $m\in \Z$.
	Since $q_1=c \geq 1=q_0$, we see by the same argument as Lemma \ref{lem.p} (1)
	that $q_{m+1} \geq q_m$ for all $m \geq 0$; in particular,
	$q_m>0$ for all $m \geq 0$.
	Since $k/l<a-1$ by \eqref{eq.m3}, and $c<k/l$ by \eqref{eq.c},
	we see that
	\begin{equation*}
		q_{-1}-q_0=(a-c)-1=(a-1)-c \geq \dfrac{k}{l}- \dfrac{k}{l}=0,
	\end{equation*}
	and hence $q_{-1}\geq q_0$.
	By the same argument as Lemma \ref{lem.p} (2),
	we see that $q_{m-1}\geq q_m$ for all $m\leq -1$; in particular,
	$q_m>0$ for all $m\leq -1$.

	Next, let us show that $q_m < p_m $ for all $m\in \Z$.
	If we set $s_m:=p_m - q_m $ for $m\in \Z$, then
	we have
	\begin{equation*}
		s_0=p_0 - q_0, \quad
		s_1=p_1 - q_1, \quad
		s_{m+2}=
		\begin{cases}
			bs_{m+1}-s_m & \text{if} \ m \ \text{is even}, \\
			as_{m+1}-s_m & \text{if} \ m \ \text{is odd}. \\
		\end{cases}
	\end{equation*}
	By the same argument as above,
	it suffices to show that  $s_1 \geq s_{0}$ and  $s_0 \leq s_{-1}$.
	First, we show that $s_1-s_0=(k-c)-(l-1)\geq 0$.
	Since $c<{k/l}$ by (\ref{eq.c}), and $l<k$ by \eqref{eq.m3},
	we have
\begin{equation*}
	(k-c)-(l-1) >k- \dfrac{k}{l}-l+1
					=k \left(  1-\dfrac{1}{l}  \right) -l+1
					>l \left( 1-\dfrac{1}{l} \right) -l+1 =0.
\end{equation*}
	Next, we show that $s_{-1}-s_0=(al-k-a+c)-(l-1)\geq 0$.
	Note that (\ref{eq.m3}) implies $(k+l)/l <a$, and (\ref{eq.c}) implies
	$(k-l)/l < c$.
	Then we compute
	\begin{align*}
		(al-k-a+c)-(l-1) 	&= a(l-1)-k+c-l+1\\
						&>  \dfrac{k+l}{l} (l-1)-k+ \dfrac{k-l}{l} -l+1 =-1.
	\end{align*}
	Because $s_{-1}-s_{0}$ is an integer, it follows that $s_{-1}\geq s_{0}$.

	Finally, the equality $q_m p_{m+1}- q_{m+1} p_{m} =k-lc$  for $m \in \Z $ can be
	easily shown by induction on $ | m | $.
\end{proof}
For $m\in \Z$ and $n\in \Z_{\geq 1}$, we say that $\pi =(\nu_1, \nu_2, \ldots \nu_u;
\sigma_0, \sigma_1, \ldots, \sigma_u ) \in \B(\lambda)$ satisfies the condition
$C(m, n)$ if $u\geq 2n+1$, and
there exists $v \in \Z$  such that $n<v<u-n+1$,
$\nu_{v}=x_{m}\lambda$, and
	$\sigma_{v+s}={{q_{m-s}}/{p_{m-s}}}$
	for $s=-n, -n+1, \ldots, n-1$;
	in this case,
we see from Lemma \ref{lem.prime} (3), along with $\nu_{v}=x_{m}\lambda$, that
\begin{equation}\label{eq.nu}
	\nu_{v+s}=x_{m-s}\lambda
	\quad \text{for} \quad v+s=1, 2,  \ldots, u.
\end{equation}
Thus, $\pi$ is of the form:
\begin{align*}
	\pi =
	\biggl(
	x_{m+v-1}\lambda, x_{m+v-2}\lambda, &\ldots, x_{m+v-u}\lambda;
	\biggr.
	\\
 	\biggl.
	\sigma_0, \sigma_1, \ldots, &\sigma_{v-n-1}, \dfrac{q_{m+n}}{p_{m+n}}, \dfrac{q_{m+n-1}}{p_{m+n-1}}, \ldots, \dfrac{q_{m-n+1}}{p_{m-n+1}}, \sigma_{v+n},  \ldots, \sigma_u
	\biggr).
\end{align*}
%
%
We set
\begin{equation}\label{eq.j}
	j:=
	\begin{cases}
		2 & \text{if} \ m \ \text{is even}, \  \\
		1  & \text{if} \ m \ \text{is odd}, \
	\end{cases}
	\quad
	j':=
	\begin{cases}
		1 & \text{if} \ m \ \text{is even}, \  \\
		2  & \text{if} \ m \ \text{is odd}; \
	\end{cases}
\end{equation}
note that
\begin{equation}\label{eq.wtp}
	\langle x_{m+v-h}\lambda, \alpha_j^\vee \rangle =
	\begin{cases}
			-p_{m+v-{h}} &\text{if } {h} \equiv v \ \mathrm{mod} \  2,\\
			p_{m+v-{h}+1} &\text{if } {h} \not \equiv v \ \mathrm{mod} \ 2,
	\end{cases}
\end{equation}
by \eqref{eq.nu} and Lemma \ref{lem.xm}.

\begin{lemma}\label{lem.frag}
	If $\pi =(x_{m+v-1}\lambda, x_{m+v-2}\lambda, \ldots,
	 x_{m+v-u}\lambda;
	\sigma_0, \sigma_1, \ldots, \sigma_u ) \in \B(\lambda) $
	satisfies  the condition $C(m, n)$,
	then for each $r= v-n, v-n+1, \ldots,  v+n-1$,
	\begin{itemize}
		\item[$(\mathrm{i})$] $H_j^\pi(\sigma_{{r}})=H_j^\pi(\sigma_v) \in \Z$
		if  ${r} \equiv v$ $\mathrm{ mod }$ $2$,
		\item[$(\mathrm{ii})$]
		$H_j^\pi(\sigma_v) < H_j^\pi(\sigma_{{r}}) < H_j^\pi(\sigma_{v})+1$
if  ${r} \not \equiv v$  $\mathrm{ mod }$ $2$;
	\end{itemize}
	in particular,
	$\{ H_j^\pi(t) \mid \sigma_{v-n}\leq t\leq \sigma_{v+n-1} \}
	\subset [H_j^\pi(\sigma_{v}) , H_j^\pi(\sigma_{v})+1  )$.
\end{lemma}

\begin{proof}
	Assume that $\pi =(x_{m+v-1}\lambda, x_{m+v-2}\lambda, \ldots, x_{m+v-u}\lambda;
	\sigma_0, \sigma_1, \ldots, \sigma_u ) \in \B(\lambda) $
	satisfies  the condition $C(m, n)$.
	For ${h}= v-n+1, v-n+2, \ldots, v+n-1$,
	we see that
	\begin{align*}
		a_{{h}}
			& :=H_j^\pi(\sigma_{{h}})-H_j^\pi(\sigma_{{h}-1}) \\
			& = \left( H_j^\pi(\sigma_{{h-1}}) +
			\langle x_{m+v-h}\lambda, \alpha_j^\vee \rangle
			(\sigma_{{h}} -\sigma_{{h-1}}) \right)-H_j^\pi(\sigma_{{h-1}}) \\
			& = \langle x_{m+v-h}\lambda, \alpha_j^\vee \rangle(\sigma_{{h}} - \sigma_{{h-1}}).
	\end{align*}
	By \eqref{eq.wtp} and the assumption that $\pi$ satisfies the condition $C(m, n)$,
	we obtain
	\begin{equation}\label{eq.am1}
		a_{{h}} =
		\begin{cases}
				-p_{m+v-{h}}\left( \dfrac{q_{m+v-{h}}}{p_{m+v-{h}}}-\dfrac{q_{m+v-{h}+1}}{p_{m+v-{h}+1}} \right)
				 &\text{if } {h} \equiv v \ \mathrm{mod} \  2,\\
				p_{m+v-{h}+1} \left( \dfrac{q_{m+v-{h}}}{p_{m+v-{h}}}-\dfrac{q_{m+v-h+1}}{p_{m+v-{h}+1}} \right)
				 &\text{if } {h} \not \equiv v \ \mathrm{mod} \  2.
		\end{cases}
	\end{equation}
	We see from  Lemma \ref{lem.pq} that
	\begin{equation}
		p_{{z}}\left( \dfrac{q_{{z}}}{p_{{z}}}-\dfrac{q_{z+1}}{p_{z+1}} \right)
		=\dfrac{k-lc}{p_{z+1}} \ \text{ and } \
		p_{z+1}\left( \dfrac{q_{{z}}}{p_{{z}}}-\dfrac{q_{z+1}}{p_{z+1}} \right)
		=\dfrac{k-lc}{p_{z}},
	\end{equation}
	for each ${z} \in \Z$.
	Here we recall that $p_{z} \geq p_0=l$ (see Remark \ref{rem.orbit}).
	By (\ref{eq.c}), we have
	\begin{equation}\label{eq.less1}
		\dfrac{k-lc}{p_{z}} \leq \dfrac{k-lc}{p_0} < 1
	\end{equation}
	for all ${z}\in\Z$.
	Combining (\ref{eq.am1})--(\ref{eq.less1}),
	we deduce that
\begin{equation}\label{eq.syou}
		0<a_{{h}}<1 \text{ if } {h} \not \equiv v  \ \mathrm{mod} \  2,  \text{ and }
		-1< a_{{h}} <0 \text{ if } {h}  \equiv v \ \mathrm{mod} \ 2.
\end{equation}

	Let $M:=\{ x\in \Z \mid v-n\leq x \leq  v+n-1\}$;
	note that $\{ v-1, v\} \subset M$ for all $n \in \Z_{\geq1 }$.
	Let  ${r} \in M$ be such that  ${r} \equiv v \ \mathrm{mod} \  2$.
	We see from Remark \ref{rem.alt}  and \eqref{int} that $H_j^\pi(\sigma_{{r}})\in \Z$.
	If  ${r}+2 \in M$ (resp., ${r}-2 \in M$),
	then we see by \eqref{eq.syou} that
	$|H_j^\pi(\sigma_{{r+2}})-H_j^\pi(\sigma_{{r}})|=|a_{{r+1}}+a_{{r}+2}|<1$
	 (resp., $|H_j^\pi(\sigma_{{r}})-H_j^\pi(\sigma_{{r}-2})|=|a_{{r-1}}+a_{{r}}|<1$).
	We see by \eqref{int} that
	$H_j^\pi(\sigma_{{r}})=H_j^\pi(\sigma_{{r}+2})$
	(resp., $H_j^\pi(\sigma_{{r}})=H_j^\pi(\sigma_{{r}-2})$).
	Hence, we obtain $H_j^\pi(\sigma_{{r}})=H_j^\pi(\sigma_{v})$
	for all ${r} \in M$ such that  ${r} \equiv v \ \mathrm{mod} \  2$.
	Thus we have shown part (i).

	Let  ${r} \in M$ be such that  ${r} \not \equiv v \ \mathrm{mod} \  2$.
	If ${r}+1 \in M$ (resp., ${r}-1 \in M$),  then we see that
	$H_j^\pi(\sigma_{{r}})=H_j^\pi(\sigma_{{r}+1})-a_{{r+1}}$
	(resp., $H_j^\pi(\sigma_{{r}})=H_j^\pi(\sigma_{{r}-1})+a_{{r}}$).
	By part $(\mathrm{i})$, we obtain
	$H_j^\pi(\sigma_{{r}})=H_j^\pi(\sigma_{v})-a_{{r+1}}$
	(resp., $H_j^\pi(\sigma_{{r}})=H_j^\pi(\sigma_{v})+a_{{r}}$).
	We see by \eqref{eq.syou} that $H_j^\pi(\sigma_v) < H_j^\pi(\sigma_{{r}}) < H_j^\pi(\sigma_{v})+1$ for all ${r} \in M$ such that
	${r} \not \equiv v \ \mathrm{mod} \  2$.
	Thus we have shown part (ii),
	thereby completing the proof of Lemma \ref{lem.frag}.
%
\end{proof}

\begin{proposition}\label{thm.cmn}
	Fix $n \geq1$. Assume that $\pi \in \B(\lambda)$ satisfies
	the condition $C(m, n)$ for some $m\in \Z$.
	Let $i \in I$.
	If $\e_i\pi\neq \mathbf{0}$, then $\e_i\pi$ satisfies the condition $C(m, n)$ or $C(m-1, n)$.
	If $\f_i\pi\neq \mathbf{0}$, then $\f_i\pi$ satisfies the condition $C(m, n)$ or $C(m+1, n)$.
\end{proposition}
\begin{proof}

	For simplicity, we prove the assertion only for the case of  $ n = 2 $.
	We set $j$ and $j'$ as \eqref{eq.j}.

		Now, assume that $\f_j\pi \neq \mathbf{0}$;
		we show that $\f_j\pi$ satisfies the condition $C(m, 2)$.
		For this, it suffices to show that $(\f_j\pi)(t)=\pi(t)$ or
		$(\f_j\pi)(t)=\pi(t)-\alpha_j$   for
		$t\in (\sigma_{v-3}, \sigma_{v+2} )= \{ t \in \R \mid \sigma_{v-3}<t< \sigma_{v+2} \}$.
	By Lemma \ref{lem.frag}, we have $H_j^\pi(\sigma_{v-2})=H_j^\pi(\sigma_{v})$,
	$H_j^\pi(\sigma_{v-2})<H_j^\pi(\sigma_{v-1})<H_j^\pi(\sigma_{v-2})+1$,
	and $H_j^\pi(\sigma_{v})<H_j^\pi(\sigma_{v+1})<H_j^\pi(\sigma_{v})+1$.
	Note that
	there exist $\sigma_{v+1} <t' \leq \sigma_{v+2}$ such that
	 $H_j^\pi(\sigma_{v-2})=H_j^\pi(\sigma_{v})=H_j^\pi(t')$
	 since  $\langle \nu_{v+2}, \alpha_j^\vee \rangle<0$
 	and $H_j^\pi(\sigma_{v+2})\in\Z$ by Remark \ref{rem.alt} and \eqref{int};
	see Figure \ref{fig.1}.
	\begin{figure}
{\unitlength 0.1in%
\begin{picture}(57.8000,27.2000)(0.2000,-29.0500)%
\put(9.0000,-17.0000){\makebox(0,0){\small{$H_j^\pi(\sigma_v)$}}}%
\put(8.0000,-11.0000){\makebox(0,0){\small{$H_j^\pi(\sigma_v)+1$}}}%
\put(12.0000,-2.5000){\makebox(0,0){$H_j^\pi(t)$}}%
%
\special{pn 8}%
\special{pa 2200 1800}%
\special{pa 1600 800}%
\special{fp}%
\special{pa 4400 1400}%
\special{pa 5200 2600}%
\special{fp}%
%
\special{pn 8}%
\special{pa 5200 2610}%
\special{pa 5400 2400}%
\special{da 0.070}%
\special{pa 1600 800}%
\special{pa 1400 1000}%
\special{da 0.070}%
\put(23.3000,-19.0000){\makebox(0,0){$t=\sigma_{v-2}$}}%
\put(29.2000,-13.0000){\makebox(0,0){$t=\sigma_{v-1}$}}%
\put(36.5000,-19.0000){\makebox(0,0){$t=\sigma_{v}$}}%
\put(45.5000,-13.0000){\makebox(0,0){$t=\sigma_{v+1}$}}%
\put(53.3000,-27.0000){\makebox(0,0){$t=\sigma_{v+2}$}}%
%
\special{pn 13}%
\special{pa 1205 2905}%
\special{pa 1205 405}%
\special{fp}%
\special{sh 1}%
\special{pa 1205 405}%
\special{pa 1185 472}%
\special{pa 1205 458}%
\special{pa 1225 472}%
\special{pa 1205 405}%
\special{fp}%
\put(17.3000,-7.0000){\makebox(0,0){$t=\sigma_{v-3}$}}%
\put(18.2000,-16.8000){\makebox(0,0){\tiny{$\langle \nu_{v-2}, \alpha_j^\vee \rangle$}}}%
\put(47.7000,-24.0000){\makebox(0,0){\tiny{$\langle \nu_{v+2}, \alpha_j^\vee \rangle$}}}%
\put(23.9000,-14.5000){\makebox(0,0){\tiny{$\langle \nu_{v-1}, \alpha_j^\vee \rangle$}}}%
\put(30.6000,-16.7000){\makebox(0,0){\tiny{$\langle \nu_{v}, \alpha_j^\vee \rangle$}}}%
\put(38.8000,-14.9000){\makebox(0,0){\tiny{$\langle \nu_{v+1}, \alpha_j^\vee \rangle$}}}%
%
\special{pn 8}%
\special{pa 5800 1800}%
\special{pa 1000 1800}%
\special{dt 0.045}%
\special{pa 1000 1200}%
\special{pa 5800 1200}%
\special{dt 0.045}%
%
\special{pn 8}%
\special{pa 2800 1400}%
\special{pa 3600 1800}%
\special{fp}%
\special{pa 4400 1400}%
\special{pa 3600 1800}%
\special{fp}%
%
\special{pn 8}%
\special{pa 2800 1400}%
\special{pa 2200 1800}%
\special{fp}%
%
\special{pn 8}%
\special{ar 4670 1800 10 10 0.0000000 6.2831853}%
\put(47.4000,-17.2000){\makebox(0,0){$t'$}}%
%
\special{pn 8}%
\special{ar 1600 800 10 10 0.0000000 6.2831853}%
%
\special{pn 8}%
\special{ar 2200 1800 10 10 0.0000000 6.2831853}%
%
\special{pn 8}%
\special{ar 2800 1400 10 10 0.0000000 6.2831853}%
%
\special{pn 8}%
\special{ar 3600 1800 10 10 0.0000000 6.2831853}%
%
\special{pn 8}%
\special{ar 4400 1400 10 10 0.0000000 6.2831853}%
%
\special{pn 8}%
\special{ar 5200 2600 10 10 0.0000000 6.2831853}%
\end{picture}}%
			\caption{\, }
			\label{fig.1}
	\end{figure}
	Let $t_0$ and $t_1$ be as \eqref{ft_0} and \eqref{ft_1}, respectively;
	note that $t_0=\sigma_{s'}$ for some $0 \leq s'\leq u$.
	By Lemma \ref{lem.frag} and the definition of $t_0$,
	we obtain
	$t_0<\sigma_{v-3}$ or $\sigma_{v+2}\leq t_0$.
	Let $m_j^\pi$ be as \eqref{eq.min}.
	If $u=5$, then we obtain $v=3$ since $2=n<v<u-n+1=4$.
	Hence, we see that $m_j^\pi=H_j^\pi(\sigma_5)$, which contradicts the
	assumption that $\f_j\pi \neq \mathbf{0}$.
	Therefore we obtain $u\geq 6$.
	If $\sigma_{v+2}\leq t_0$, then it is obvious from the definition of $\f_j$
	that $(\f_j\pi)(t)=\pi(t)$
	for $t\in (\sigma_{v-3}, \sigma_{v+2} )$.
	If $t_0<\sigma_{v-3}$, then $H_j^\pi(\sigma_{v-2})>H_j^\pi(t_0)=m_j^\pi\in\Z$
	by the definition of $t_0$.
	Note that $H_j^\pi(\sigma_{v-2}) \in \Z$ by \eqref{int}, and hence
	$H_j^\pi(\sigma_{v-2})\geq H_j^\pi(t_0)+1= m_j^\pi+1$.
	Because $H_j^\pi(\sigma_{v-3})>H_j^\pi(\sigma_{v-2})\geq m_j^\pi+1$,
	we see that $t_1<\sigma_{v-3}$.
	Therefore we obtain $(\f_j\pi)(t)=\pi(t)- \alpha_j$
	for $t\in (\sigma_{v-3}, \sigma_{v+2} )$
	by the definition of  $\f_j$.

	Assume that $\e_j\pi \neq \mathbf{0}$;
	we show that $\e_j\pi$ satisfies the condition $C(m, 2)$ or $C(m-1, 2)$.
	Take $t_1$ and  $t_0$ as \eqref{et_1} and \eqref{et_0}, respectively;
	note that $t_1=\sigma_{s'}$ for some $0 \leq s' \leq u$.
	By the definition of $t_1$ and Lemma \ref{lem.frag},
	we obtain $t_1<\sigma_{v-3}$, $\sigma_{v+2}\leq t_1$, or $ t_1=\sigma_{v-2}$.
If  $t_1<\sigma_{v-3}$, then it is obvious by the definition of $\e_j$
	that  $(\e_j\pi)(t)=\pi(t)+\alpha_j$
	for $t\in (\sigma_{v-3}, \sigma_{v+2} )$,
	and hence $\e_j\pi $ satisfies the condition $C(m, 2)$.
If $\sigma_{v+2}<t_1$, then $H_j^\pi(\sigma_{v+2})>H_j^\pi(t_1)=m_j^\pi\in\Z$
	by the definition of $t_1$.
	Note that $H_j^\pi(\sigma_{v+2})\in\Z$ by \eqref{int},
	and hence $H_j^\pi(\sigma_{v+2})\geq H_j^\pi(t_1)+1=m_j^\pi+1$.
	Because  $H_j^\pi(\sigma_{v+3})>H_j^\pi(\sigma_{v+2})\geq m_j^\pi+1$,
	we see that $\sigma_{v+3}<t_0$.
	Therefore we obtain $(\e_j\pi)(t)=\pi(t)$
	for $t\in (\sigma_{v-3}, \sigma_{v+2} )$, and hence $\e_j\pi $ satisfies the condition $C(m, 2)$.
	Assume that  $ t_1=\sigma_{v+2}$.
	Since $H_j^\pi(\sigma_v)=H_j^\pi(t') > H_j^\pi(t_1) =m^\pi_j \in \Z$ by the definition of $t_1$,
	we see that $H_j^\pi(t')\geq  H_j^\pi(t_1)+1$.
		Hence, we obtain $t'\leq t_0$.
		Therefore, we have
	\begin{equation*}
		\e_j\pi=
		\begin{cases}
			(\nu_1, \ldots, \nu_u, x_{m-3}\lambda;
			\sigma_0, \ldots, \sigma_{v+1}, t_0, 1) &
			\text{if } v=u-2,\\
			(\nu_1, \ldots, \nu_u;
			\sigma_0, \ldots, \sigma_{v+1}, t_0, \sigma_{v+3}, \ldots, \sigma_u )&
			\text{if } v < u-2,
		\end{cases}
	\end{equation*}
	which satisfies the condition $C(m,2)$.
Assume that $t_1=\sigma_{v-2}$.
If $v=3$, then it is obvious that $\sigma_0=\sigma_{v-3}\leq t_0$;
note that $\sigma_{v-3}= t_0$ if and only if $H_j^\pi(\sigma_1)=-1$.
If $v>3$, then $H_j^\pi(\sigma_{v-4})>H_j^\pi(t_1)=m_j^\pi\in\Z$
	by the definition of $t_1$.
	Note that $H_j^\pi(\sigma_{v-4})\in\Z$ by \eqref{int},
	and hence $H_j^\pi(\sigma_{v-4})\geq H_j^\pi(t_1)+1=m_j^\pi+1$.
	Because  $H_j^\pi(\sigma_{v-3})>H_j^\pi(\sigma_{v-4})\geq m_j^\pi+1$,
	we see that $\sigma_{v-3}<t_0$.
Therefore we see that
\begin{equation*}
	\e_j\pi=
	\begin{cases}
		(\nu_2, \nu_3, \ldots, \nu_u;
		\sigma_0, \sigma_{2}, \sigma_{3}, \ldots, \sigma_{u}) &
		\text{if } v=3 \text{ and } H_j^\pi(\sigma_1)=-1,\\
		(\nu_1, \ldots, \nu_u;
		\sigma_0, \ldots, \sigma_{v-3}, t_0, \sigma_{v-1}, \ldots, \sigma_u )&
		\text{otherwise}.
	\end{cases}
\end{equation*}
Also,
	we see by $H_j^\pi(\sigma_{v-2})=H_j^\pi(t')$ and
	the definition of $t_1$ that $t'=\sigma_{v+2}$.
	Since  $H_j^\pi(\sigma_{v})=H_j^\pi(\sigma_{v+2})$, we have
	\begin{equation}\label{eq.aaaa}
			a_{v+1}
		=H_{{j}}^\pi(\sigma_{v+1})-H_{{j}}^\pi(\sigma_{v+2}).
	\end{equation}
	Here we can rewrite (\ref{eq.aaaa}) as
	\begin{equation}\label{eq.eqa}
		p_{m} \left( \dfrac{q_{m-1}}{p_{m-1}}-
		\dfrac{q_m}{p_{m}} \right)
		=p_{m-2} \left( \sigma_{v+2}-
		\dfrac{q_{m-1}}{p_{m-1}} \right).
	\end{equation}
	By \eqref{eq.pm}, \eqref{eq.pm2}, and \eqref{eq.qm}, we see that
	\begin{equation*}
		\sigma_{v+2}
		=\dfrac{1}{p_{m-2}}\left( \dfrac{(p_{m}+p_{m-2})q_{m-1}}{p_{m-1}}-q_m
		\right) =\dfrac{q_{m-2}}{p_{m-2}}.
	\end{equation*}
	Since $\sigma_{v+2} =q_{m-2}/p_{m-2}<1$ by Lemma \ref{lem.pq},
	we obtain $v+2<u$.
	Write $\e_j\pi$ as: $\e_j\pi=(\nu'_1, \ldots, \nu'_{u'}; \sigma'_0, \ldots, \sigma'_{u'})$.
	If $v=3$ and $H_j^\pi(\sigma_1)=-1$,
	then $u'=u-1, \nu'_s=\nu_{s+1}$ for $s=1, \ldots u'$,
	and $\sigma'_0=0$, $\sigma'_s=\sigma_{s+1}$ for $s=1, \ldots u'$.
	We set $v':=v$. Then we obtain
	$2<v' <u'-1$, $\nu'_{v'}=x_{m-1}\lambda$,
	and $\sigma'_{v'+s}=q_{(m-1)-s}/p_{(m-1)-s}$ for $s=-2, -1, 0, 1$.
	Hence we see that $\e_j\pi$ satisfies the condition $C(m-1, 2)$.
	If $v>3$, or $v=3$ and $H_j^\pi(\sigma_1)\neq -1$,
	then $u'=u, \nu'_s=\nu_{s}$ for $s=1, \ldots u'$,
	and $\sigma'_{v-2}=t_0$, $\sigma'_s=\sigma_{s}$ for $s=1, \ldots, v-3, v-1, \ldots, u'$.
	Hence, we see that $\e_j\pi$ satisfies the condition $C(m-1, 2)$ with $v':=v+1$.

	Similarly, we can show that if $\e_{j'}\pi\neq\mathbf{0}$ (resp., $\f_{j'}\pi\neq\mathbf{0}$), then it satisfies the condition $C(m, 2)$ (resp., $C(m, 2)$ or $C(m+1, 2)$).
	Thus we have proved Proposition \ref{thm.cmn}.
	\end{proof}
	Now, for each $n\in\Z_{\geq1}$, we define
	\begin{align*}
		\underline{\nu^{(n)}}:
		& \  x_n\lambda, \ldots ,  x_1\lambda, x_0\lambda, x_{-1}\lambda, \ldots ,x_{-n}\lambda,\\
		{\underline{\sigma^{(n)}}:}
		& \  { 0< \frac{q_{n}}{p_{n}} <\cdots< \frac{q_{1}}{p_{1}}< \frac{q_{0}}{p_{0}}<
		 \frac{q_{-1}}{p_{-1}}< \cdots< \frac{q_{-n+1}}{p_{-n+1}}<  1;}
	\end{align*}
	{the inequalities in $\underline{\sigma^{(n)}}$ follow from Lemma \ref{lem.pq}}.
	We see that
	$\pi^{(n)}:= (\underline{\nu^{(n)}}, \underline{\sigma^{(n)}})$ is
	an LS path of shape $\lambda$
	satisfying the condition $C(0, n)$.
	We denote by $\B(\lambda; \pi^{(n)})$ the connected component of $\B(\lambda)$
	containing $\pi^{(n)}$;
	note that an element of $\B(\lambda; \pi^{(n)})$ satisfies
	the condition $C(m, n)$ for some $m\in\Z$ by Proposition \ref{thm.cmn}.
	Hence we have the following corollary, which proves Proposition \ref{thm.main2}.

\begin{corollary}\label{cor.main2}
	If $n\neq n'$,
	then $\B(\lambda; \pi^{(n)})\cap \B(\lambda; \pi^{(n')})=\emptyset$.
	In particular,
	the crystal graph of $\B(\lambda)$
	has infinitely many connected
	components.
\end{corollary}

\begin{proof}
	We may assume that $n'<n$.
	Suppose, for a contradiction, that $ \B(\lambda; \pi^{(n)})= \B(\lambda; \pi^{(n')})$.
	For $\pi=(\nu_1, \ldots, \nu_u; \sigma_0, \ldots, \sigma_u) \in \B(\lambda)$,
	we define $\ell(\pi):=u$.
	By Proposition \ref{thm.cmn}, we see that
	if $\pi\in \B(\lambda; \pi^{(n)})$, then $\ell (\pi) \geq 2n+1$.
	Hence, we have  $\ell (\pi^{(n')}) \geq 2n+1$ since $\pi^{(n')} \in \B(\lambda; \pi^{(n)})$.
	However,  we have   $\ell (\pi^{(n')})=2n'+1 < 2n+1$ by the definition of $\pi^{(n')}$
	{and $n'<n$,} which is a contradiction.
	Thus we have proved Corollary \ref{cor.main2}.
	\end{proof}

\subsection{Proof of Proposition \ref{thm.main3}.}\label{S.prf.main3}
Assume that
$\lambda=k\Lambda_1-l\Lambda_2\in P$ is
of the form either $(\mathrm{i})$ or $(\mathrm{ii})$ in Theorem \ref{thm.M},
and $k$ and $l$ are not relatively prime.
Let $d\geq 2$ be the greatest common divisor of $k$ and $l$.
We set $k':=k/d$, $l':=l/d$, and
 $\lambda':=(1/d)\lambda= k'\Lambda_1-l'\Lambda_2 \in P$.
Define the  sequence $\{ p_m\}_{m\in\Z}$ by
\eqref{eq.pm} and  \eqref{eq.pm2} for $\lambda'$:
\begin{equation}
	p_0=l', \quad
	p_1=k', \quad
	p_{m+2}=
	\begin{cases}
		bp_{m+1}-p_m & \text{if} \ m \ \text{is even}, \\
		ap_{m+1}-p_m & \text{if} \ m \ \text{is odd}. \\
	\end{cases}
\end{equation}
Then we see from Lemma \ref{lem.prime} (1)
that  $p_m$ and $p_{m+1}$ are relatively prime for all $m\in\Z$.
Note that
\begin{equation}\label{eq.2xm}
	x_m\lambda=
	\begin{cases}
		d p_{m+1}\Lambda_1-d p_m\Lambda_2 &
		\text{if}\ m \ \text{is even},\\
		-d p_m\Lambda_1+d p_{m+1}\Lambda_2 &
		\text{if}\ m \ \text{is odd},\\
	\end{cases}
\end{equation}
for $m \in \Z$.

We first show that the proof of
Proposition \ref{thm.main3} is reduced to the case that $d=2$,
i.e., $\lambda=2\lambda'$.
For this purpose, we begin by recalling the notion of a concatenation of LS paths
(in a general setting).
Let $\mu \in P$ be an arbitrary integral weight,
and $m \in \Z_{\geq 1}$.
For $\pi_1, \pi_2, \ldots, \pi_m \in \B(\mu)$,
we define a concatenation
$\pi_1 \ast \pi_2 \ast \cdots \ast \pi_m$
of them by:
\begin{align*}
	(\pi_1 \ast \pi_2 \ast \cdots \ast \pi_m )(t)
	= \sum_{l=1}^{k-1}\pi_l(1) &+ \pi_k(mt-k+1)\\
	& \text{ for } \frac{k-1}{m}\leq t \leq \frac{k}{m},
	\ 1\leq k \leq m,
\end{align*}
and set
\begin{equation*}
	\B(\mu)^{\ast m}
	=\underbrace{\B(\mu) \ast \cdots \ast \B(\mu)}_{m \text{ times}}
	:=\{ \pi_1 \ast \cdots \ast \pi_m \mid \pi_k \in \B(\mu),
	\ 1\leq k \leq m  \}.
\end{equation*}
We endow $\B(\mu)^{\ast m}$ with a crystal structure as follows.
Let $\pi= \pi_1 \ast \cdots \ast \pi_m \in \B(\mu)^{\ast m}$.
First we define
$\mathrm{wt}(\pi) := \pi(1) = \pi_1(1)+ \cdots+ \pi_m(1)$;
notice that $\pi(1) \in P$ since $\pi_k(1) \in P$
for all $1 \leq k \leq m$.
Next, for $i \in I$, we define $\e_i \pi$ and $\f_i \pi$
in exactly the same way as
for elements in $\B(\mu)$ (see \S \ref{section.LSpath});
notice that the condition \eqref{int} holds for every element
in $\B(\mu)^{\ast m}$.
We deduce that if $\e_i \pi \neq \mathbf{0}$,
then $\e_i \pi
= \pi_1 \ast \cdots \ast \e_i \pi_k \ast \cdots \ast \pi_m$
for some $1 \leq k \leq m$;
the same holds also for $\f_i$.
Therefore the set $\B(\mu)^{\ast m} \cup \{\mathbf{0}\}$
is stable under the action of $\e_i$ and $\f_i$ for $i \in I$. Finally, for $i \in I$,
we set $ \varepsilon_i (\pi):= \mathrm{max} \{ n \geq 0 \mid \e_i ^n \pi \neq \mathbf{0} \}$ and
$ \varphi_i (\pi):= \mathrm{max} \{ n \geq 0 \mid \f_i ^n \pi \neq \mathbf{0} \}$.
 We know the following proposition from \cite[\S 2]{L}.
\begin{proposition}
	Let $\mu \in P$ be an arbitrary integral weight,
	and $m \in \Z_{\geq 1}$.
	The set $\B(\mu)^{\ast m}$ together with the maps
	$\mathrm{wt} :\B(\mu)^{\ast m} \rightarrow P $,
	$\e_i, \f_i  :\B(\mu)^{\ast m} \rightarrow \B(\mu)^{\ast m} \cup \{\mathbf{0}\} $, $i \in I$, and
	$ \varepsilon_i$, $\varphi_i :\B(\mu)^{\ast m} \rightarrow \Z_{\geq0}$, $i \in I$, is a crystal.
	Moreover, the map
	\begin{equation*}
		\B(\mu)^{\ast m} \rightarrow\B(\mu)^{\otimes m}, \
		\pi_1 \ast \cdots \ast \pi_m \mapsto
		\pi_1 \otimes \cdots \otimes \pi_m,
	\end{equation*}
	is an isomorphism of crystals.
\end{proposition}
Let $\pi =(\nu_1, \nu_2, \ldots, \nu_u;
\sigma_0, \sigma_1, \ldots, \sigma_u ) \in \B(m\mu)$.
For each $1\leq k\leq m$,
let $s, s'$ be such that $\sigma_{s-1}\leq (k-1)/m< \sigma_{s}$ and
$\sigma_{s'-1}< k/m\leq \sigma_{s'}$, respectively.
We set
\begin{equation*}
	\pi_k:=
	\left(
	\frac{1}{m}\nu_s, \frac{1}{m}\nu_{s+1},
	\ldots, \frac{1}{m}\nu_{s'};
	0, m\sigma_s-k+1,
	m\sigma_{s+1}-k+1,
	\ldots, m\sigma_{s'-1}-k+1, 1
	 \right).
\end{equation*}
By the definition of LS paths and the assumption that
$\pi \in \B(m\mu)$,
we deduce that $\pi_k \in \B(\mu)$.
Moreover, it is easy to check that
$\pi =\pi_1 \ast \cdots \ast \pi_m$ since
\begin{equation*}
	\pi_k(t)=\pi \left( \frac{1}{m}t+ \frac{k-1}{m}  \right)
	-\pi \left(\frac{k-1}{m}  \right) \
	\text{for } t \in [0, 1].
\end{equation*}
Therefore, it follows  that $\B(m\mu)$ is contained in $\B(\mu)^{\ast m}$,
and hence is a subcrystal of
$\B(\mu)^{\ast m} \cong\B(\mu)^{\otimes m}$
consisting of the elements $\pi_1 \otimes \pi_2 \otimes \cdots \otimes \pi_m$
such that $\kappa(\pi_p) \geq \iota(\pi_{p+1})$ for all $1\leq p \leq m-1$.


Now, we return to the proof of Proposition \ref{thm.main3}.
The map
\begin{equation*}
	\B(\lambda')^{\otimes d} \rightarrow  \B(\lambda')^{\otimes 2}, \
	\pi_1 \otimes \pi_2 \otimes \cdots  \otimes  \pi_d \mapsto\pi_1 \otimes \pi_2,
\end{equation*}
induces a surjective map $\Phi $ from
$\B(d\lambda') \subset \B(\lambda')^{\ast d}
\cong \B(\lambda')^{\otimes d}$
to $\B(2\lambda') \subset \B(\lambda')^{\ast 2}
\cong \B(\lambda')^{\otimes 2}$;
note that $\Phi $ is not necessarily a morphism of crystals.
It follows that the inverse image
$\Phi^{-1}(C)\subset \B(d\lambda')$
of a connected component $C$ of $\B(2\lambda')$
is a subcrystal of $\B(d\lambda')$. This shows that if
$\B(2\lambda')$ has infinitely many connected components,
then so does $\B(d\lambda')$.
Therefore our proof of Proposition \ref{thm.main3} is reduced to the case that $d=2$,
i.e., $\lambda=2\lambda'$.
In \cite{O}, it has been shown that the crystal graph of
$\B(2\Lambda_1-2\Lambda_2)$
has infinitely many connected components.
We can prove Proposition \ref{thm.main3} with $d=2$ in exactly the same way as it.
However, since \cite{O} is written in Japanese,
we write the proof 
also here for completion.

\begin{lemma}\label{lem.o1}
	Let $m, n \in \Z $ be such that $m-n\geq 2$,
	and let $0<\sigma<1$ be a rational number.
	There exists a $\sigma$-chain for  $(x_{m}\lambda, x_{n}\lambda )$
	if and only if $\sigma=1/2$.
\end{lemma}

\begin{proof}
	Assume that there exists a $\sigma$-chain
	$ x_m\lambda=\nu_0> \nu_1> \cdots> \nu_u=x_n\lambda$
	 for  $(x_{m}\lambda, x_{n}\lambda )$.
	We see from  Proposition \ref{prop.Hasse} that
	$\nu_v=x_{m-v}\lambda$ for $v=0, 1, \ldots, m-n$.
	We set
	\begin{equation*}
		i:=
		\begin{cases}
			2 & \text{if} \ m \ \text{is even}, \  \\
			1  & \text{if} \ m \ \text{is odd}, \
		\end{cases}
		\quad
		j:=
		\begin{cases}
			1 & \text{if} \ m \ \text{is even}, \  \\
			2  & \text{if} \ m \ \text{is odd}. \
		\end{cases}
	\end{equation*}
	By (\ref{eq.2xm}), it follows that
	$\sigma \langle x_{m}\lambda, \alpha_i^\vee \rangle= -2\sigma p_{m} $ and
	$\sigma \langle x_{m-1}\lambda, \alpha_j^\vee \rangle= -2\sigma p_{m-1} $.
	By the definition of a $\sigma$-chain and the assumption that $m-n\geq 2$,
	both $-2\sigma p_m $ and  $-2\sigma p_{m-1} $ are integers.
	Since $p_m$ and $p_{m-1}$ are relatively prime,
	we obtain $\sigma =1/2$.

	Conversely, we assume that $\sigma =1/2$.
	By Proposition \ref{prop.Hasse} and (\ref{eq.2xm}), it is obvious that
	the sequence
	$x_{m}\lambda, x_{m-1}\lambda, \ldots, x_{n+1}\lambda, x_{n}\lambda$
	becomes a $\sigma$-chain for $(x_m\lambda, x_n\lambda)$.
\end{proof}

\begin{lemma}\label{lem.o2}
	Let $\pi$ be an LS path of shape $\lambda$.
	It holds that $\pi(1/2)\in P$.
\end{lemma}

\begin{proof}
	Let $\pi=(\nu_1,  \ldots, \nu_{s-1}, \nu_{s}, \ldots,  \nu_u;
	\sigma_0,  \ldots, \sigma_{s-1}, \sigma_{s}, \ldots, \sigma_u) \in \B(\lambda)$,
	and assume that  $\sigma_{s-1}<1/2\leq \sigma_{s}$.
	Then we compute
	\begin{align*}
		\pi\left(\dfrac{1}{2} \right) &=\sum_{v=1}^{s-1}(\sigma_{v}- \sigma_{v-1})\nu_v
		+\left(\dfrac{1}{2}-\sigma_{s-1}\right) \nu_s \\
		&= \sum_{v=1}^{s-1}\sigma_{v}(\nu_v-\nu_{v+1})+\dfrac{1}{2}\nu_s.
	\end{align*}
	Since there exists a $\sigma_v$-chain for $(\nu_v, \nu_{v+1})$ for each $1\leq v\leq s-1$,
	we see that $\sum_{v=1}^{s-1}\sigma_{v}(\nu_v-\nu_{v+1}) \in P$
	(cf. \cite[\S 4]{L}).
	Moreover, (\ref{eq.2xm}) implies $(1/2)\nu_s \in P$.
	Thus we obtain $\pi(1/2)\in P$, as desired.
\end{proof}

For each $\nu \in W\lambda $,
there exists unique $m\in \Z$ such that $\nu=x_m\lambda$. Then
we define  $z(\nu):=m$.
For $r\in \Z_{\geq0}$, we define a subset $\B_r(\lambda)$ of $\B(\lambda)$
as follows: If $r=0$, then we set
\begin{equation*}
	\B_0(\lambda):=\{ (\nu_1, \ldots, \nu_u;
	\sigma_0, \ldots, \sigma_u) \in \B(\lambda) \mid
	z(\nu_v)-z(\nu_{v+1})=1 \text{ for } v=1, \ldots, u-1 \}.
\end{equation*}
If $r\geq1$, then we set
\begin{equation*}
	\B_r(\lambda):=\{\pi= (\nu_1,  \ldots, \nu_u;
	\sigma_0, \ldots, \sigma_u) \in \B(\lambda) \mid
	\pi	\text{ satisfies the condition  } (\mathrm{I}) \text{ or }  (\mathrm{\II}) \},
\end{equation*}
where
\begin{align*}
	(\mathrm{I}): & \text{ There exists } 1\leq s \leq u-1
	\text{ such that  } z(\nu_s)-z(\nu_{s+1}) =2r; \\
	(\mathrm{\II}): &  \text{ There exists } 1\leq s \leq u-1
	\text{ such that  } z(\nu_s)-z(\nu_{s+1}) =2r+1.
\end{align*}
\begin{remark}\label{rem.o}
	Let $(\nu_1, \ldots, \nu_u;
	\sigma_0,  \ldots, \sigma_u) \in \B(\lambda) $.
	Assume that  there exists $ 1\leq s \leq u-1$
	such that  $z(\nu_s)-z(\nu_{s+1}) \geq2$.
	Then, we see by Lemma \ref{lem.o1} that $\sigma_s=1/2$ and
	$z(\nu_v)-z(\nu_{v+1}) =1$ for each $v=1, 2, \ldots, s-1, s+1, \ldots, u-1$.
	Hence, we obtain $\B(\lambda)=\bigsqcup_{r\in\Z_{\geq0}}\B_r(\lambda)$.
\end{remark}
Proposition \ref{thm.main3} is
a corollary of the following theorem.


\begin{theorem}\label{thm.Om}
	Let $r\in \Z_{\geq 0}$. Let $\pi \in \B_r(\lambda)$, and  $i \in I$.
	If $\e_i \pi \neq \mathbf{0}$, then $\e_i \pi \in \B_r(\lambda)$.
	If $\f_i \pi \neq \mathbf{0}$, then $\f_i \pi \in \B_r(\lambda)$.
	Therefore, $\B_r(\lambda)$ is a subcrystal of $\B(\lambda)$ for each $r\in\Z_{\geq0}$.
\end{theorem}

\begin{proof}
	The proof  is divided into three cases.

	$\bf{Case \ 1}$: Assume that $r\geq 1$, and $\pi=(\nu_1,  \ldots, \nu_u;
	\sigma_0,  \ldots, \sigma_u) \in \B_r(\lambda)$ satisfies the condition (I).
	We set $z(\nu_{s})=m$, $z(\nu_{s+1})=n$; note that $m-n=2r$. We set
	\begin{equation*}
		j:=
		\begin{cases}
			2 & \text{if} \ m \ \text{is even}, \  \\
			1  & \text{if} \ m \ \text{is odd}, \
		\end{cases}
		\quad
		{j'}:=
		\begin{cases}
			1 & \text{if} \ m \ \text{is even}, \  \\
			2  & \text{if} \ m \ \text{is odd}. \
		\end{cases}
	\end{equation*}
	Then we see by \eqref{eq.2xm} that
	$\langle \nu_s, \alpha_j^\vee \rangle <0$,
	$\langle \nu_{s}, \alpha_{j'}^\vee \rangle >0$.
	Moreover, we see that
	$\langle \nu_{s+1}, \alpha_j^\vee \rangle <0$,
	$\langle \nu_{s+1}, \alpha_{j'}^\vee \rangle >0$
	because $m-n \in 2\Z$.

First,
	let us show that $\e_j\pi\in \B_r(\lambda)$ if $\e_j\pi \neq \mathbf{0}$.
	Take  $t_1$ and $t_0$ as (\ref{et_1}) and (\ref{et_0}), respectively
	(with $i$ replaced by $j$);
	note that $t_1=\sigma_v$ for some $0\leq v\leq u$.
	Since $\langle \nu_s, \alpha_j^\vee \rangle <0$
	and $\langle \nu_{s+1}, \alpha_j^\vee \rangle <0$ as seen above,
	the function $H_j^\pi(t)$ does not attain its minimum value at $t=\sigma_{s-1}, \sigma_{s}$
	(see the left figure in Figure \ref{fig.2}).
	Thus
	we obtain $t_1 \neq \sigma_{s-1}, \sigma_s$.
	If $t_1<\sigma_{s-1}$, then the assertion is obvious by
	the definition of $\e_j$ and Remark \ref{rem.o}.
	Assume that  $t_1 \geq  \sigma_{s+1}$.
	Since $H_j^\pi(\sigma_s)>  H_j^\pi(t_1)= m_j^\pi \in \Z $ by the definition of $t_1$,
	and since  $H_j^\pi(\sigma_s)=  H_j^\pi(1/2) \in \Z $ by Lemma \ref{lem.o2} and Remark \ref{rem.o},
	we see that $H_j^\pi(\sigma_s) \geq H_j^\pi(t_1)+1$.
	Therefore, we have $t_0 \geq \sigma_s$ by the definition of $t_0$.
	If $t_0> \sigma_s$, then it is obvious
	by the definition of $\e_j$ and Remark \ref{rem.o} that $\e_j\pi\in \B_r(\lambda)$.
	If $t_0= \sigma_s$, then we deduce by the definition of $\e_j$ that
	$\e_j\pi$ is of the form
	\begin{equation*}
		\e_j\pi =
		\begin{cases}
			(\nu_1,  \ldots, \nu_s, r_j\nu_{s+1};
			\sigma_0,  \ldots,  \sigma_s, \sigma_u)
			& \text{ if } s = u-1 \text{ or } u-2, \\
			(\nu_1,  \ldots, \nu_s, r_j\nu_{s+1}, \ldots,  \nu_u;
		\sigma_0,  \ldots, \sigma_{s}, \sigma_{s+2},  \ldots, \sigma_u)
			& \text{ if } s \leq u-3.
		\end{cases}
	\end{equation*}
	Since $r_j\nu_{s+1}=r_j x_{n}\lambda =x_{n-1}\lambda$ and $m-(n-1) =2r+1$, we obtain
	$\e_j\pi\in \B_r(\lambda)$.

Next, let us show that $\e_{j'}\pi\in \B_r(\lambda)$ if $\e_{j'}\pi \neq \mathbf{0}$.
Take  $t_1$ and $t_0$ as (\ref{et_1}) and (\ref{et_0}), respectively
(with $i$ replaced by $j'$).
Since $\langle \nu_s, \alpha_{j'}^\vee \rangle >0$
and $\langle \nu_{s+1}, \alpha_{j'}^\vee \rangle >0$ as seen above,
the function  $H_{j'}^\pi(t)$ does not attain its minimum value at $t=\sigma_s, \sigma_{s+1}$
(see the right figure in Figure \ref{fig.2}).
Thus
	we obtain $t_1\neq \sigma_{s}, \sigma_{s+1}$.
If $t_1\leq \sigma_{s-1}$,
then the assertion is obvious
by the definition of $\e_{j'}$ and Remark \ref{rem.o}.
	If $t_1\geq \sigma_{s+2}$, then
	$H_{j'}^\pi(\sigma_{s-1}) >  H_{j'}^\pi(t_1)= m_{j'}^\pi \in \Z$ by the definition of $t_1$.
	Notice that $H_{j'}^\pi(\sigma_{s-1}) \in \Z$ by \eqref{int},
	and hence $H_{j'}^\pi(\sigma_{s-1}) \geq  H_{j'}^\pi(t_1)+1= m_{j'}^\pi +1$.
	Because $H_{j'}^\pi(\sigma_{s+1}) >  H_{j'}^\pi(\sigma_{s-1}) \geq m_{j'}^\pi +1$,
	we see that
	$\sigma_{s+1}<t_0$.
	Therefore, $\e_{j'}\pi\in \B_r(\lambda)$
	by the definition of $\e_{j'}$ and Remark \ref{rem.o}.

	Similarly, we can show (in Case 1) that if $\f_i \pi \neq \mathbf{0}$ for $i \in I$,
	then $\f_i \pi \in \B_r(\lambda)$
	\begin{figure}
		\centering
{\unitlength 0.1in%
\begin{picture}(55.5000,16.7000)(2.5000,-18.0000)%
%
\special{pn 8}%
\special{pa 1195 600}%
\special{pa 1995 800}%
\special{fp}%
\special{pa 1995 800}%
\special{pa 2595 1600}%
\special{fp}%
%
\special{pn 8}%
\special{pa 4000 1395}%
\special{pa 4800 795}%
\special{fp}%
\special{pa 4800 795}%
\special{pa 5400 595}%
\special{fp}%
%
\special{pn 13}%
\special{ar 1995 800 10 10 0.0000000 6.2831853}%
%
\special{pn 13}%
\special{ar 4800 795 10 10 0.0000000 6.2831853}%
\put(5.9500,-2.0000){\makebox(0,0){$H_j^\pi(t)$}}%
\put(34.0000,-1.9500){\makebox(0,0){$H_{j'}^\pi(t)$}}%
\put(20.4500,-6.9000){\makebox(0,0){$t=\sigma_{s}$}}%
\put(12.9500,-4.9000){\makebox(0,0){$t=\sigma_{s-1}$}}%
\put(26.8500,-16.9000){\makebox(0,0){$t=\sigma_{s+1}$}}%
\put(55.6000,-4.8500){\makebox(0,0){$t=\sigma_{s+1}$}}%
\put(47.8000,-6.5500){\makebox(0,0){$t=\sigma_{s}$}}%
\put(41.2000,-14.8500){\makebox(0,0){$t=\sigma_{s-1}$}}%
%
\special{pn 8}%
\special{pa 1195 600}%
\special{pa 995 700}%
\special{fp}%
\special{pa 2595 1600}%
\special{pa 2795 1400}%
\special{fp}%
%
\special{pn 8}%
\special{pa 2795 1400}%
\special{pa 2895 1300}%
\special{dt 0.045}%
\special{pa 995 700}%
\special{pa 795 800}%
\special{dt 0.045}%
%
\special{pn 8}%
\special{pa 5400 595}%
\special{pa 5600 795}%
\special{fp}%
\special{pa 4000 1395}%
\special{pa 3800 1295}%
\special{fp}%
%
\special{pn 8}%
\special{pa 3800 1295}%
\special{pa 3600 1195}%
\special{dt 0.045}%
\special{pa 5600 795}%
\special{pa 5800 995}%
\special{dt 0.045}%
%
\special{sh 0.300}%
\special{ia 5400 595 10 10 0.0000000 6.2831853}%
\special{pn 8}%
\special{pn 8}%
\special{pa 5410 595}%
\special{pa 5405 604}%
\special{fp}%
\special{pa 5410 595}%
\special{pa 5410 595}%
\special{fp}%
%
\special{sh 0.300}%
\special{ia 4000 1395 10 10 0.0000000 6.2831853}%
\special{pn 8}%
\special{pn 8}%
\special{pa 4010 1395}%
\special{pa 4005 1404}%
\special{fp}%
\special{pa 4010 1395}%
\special{pa 4010 1395}%
\special{fp}%
%
\special{pn 8}%
\special{pn 8}%
\special{pa 2605 1600}%
\special{pa 2600 1609}%
\special{fp}%
\special{pa 2605 1600}%
\special{pa 2605 1600}%
\special{fp}%
%
\special{pn 8}%
\special{pn 8}%
\special{pa 1205 600}%
\special{pa 1200 609}%
\special{fp}%
\special{pa 1205 600}%
\special{pa 1205 600}%
\special{fp}%
%
\special{pn 13}%
\special{pa 595 1800}%
\special{pa 595 300}%
\special{fp}%
\special{sh 1}%
\special{pa 595 300}%
\special{pa 575 367}%
\special{pa 595 353}%
\special{pa 615 367}%
\special{pa 595 300}%
\special{fp}%
%
\special{pn 13}%
\special{pa 3395 1800}%
\special{pa 3395 300}%
\special{fp}%
\special{sh 1}%
\special{pa 3395 300}%
\special{pa 3375 367}%
\special{pa 3395 353}%
\special{pa 3415 367}%
\special{pa 3395 300}%
\special{fp}%
\put(14.5500,-7.9000){\makebox(0,0){\tiny{$\langle \nu_{s}, \alpha_j^\vee \rangle$}}}%
\put(24.9500,-11.0000){\makebox(0,0){\tiny{$\langle \nu_{s+1}, \alpha_j^\vee \rangle$}}}%
\put(42.3000,-10.0500){\makebox(0,0){\tiny{$\langle \nu_{s}, \alpha_{j'}^\vee \rangle$}}}%
\put(51.6000,-8.1500){\makebox(0,0){\tiny{$\langle \nu_{s+1}, \alpha_{j'}^\vee \rangle$}}}%
%
\special{pn 8}%
\special{ar 1195 600 10 10 0.0000000 6.2831853}%
%
\special{sh 0.300}%
\special{ia 2595 1600 10 10 0.0000000 6.2831853}%
\special{pn 8}%
\special{ar 2595 1600 10 10 0.0000000 6.2831853}%
%
\special{pn 8}%
\special{ar 4000 1395 10 10 0.0000000 6.2831853}%
%
\special{pn 8}%
\special{ar 5400 595 10 10 0.0000000 6.2831853}%
\end{picture}}%
			\caption{\, }
			\label{fig.2}
	\end{figure}

	$\bf{Case \ 2}$: Assume that  $r\geq 1$, and $\pi=(\nu_1,  \ldots, \nu_u;
	\sigma_0,  \ldots, \sigma_u) \in \B_r(\lambda)$ satisfies the condition (\II).
	We set $z(\nu_{s})=m$, $z(\nu_{s+1})=n$; note that $m-n=2r+1$.
	We set
	\begin{equation*}
		j:=
		\begin{cases}
			2 & \text{if} \ m \ \text{is even}, \  \\
			1  & \text{if} \ m \ \text{is odd}, \
		\end{cases}
		\quad
		{j'}:=
		\begin{cases}
			1 & \text{if} \ m \ \text{is even}, \  \\
			2  & \text{if} \ m \ \text{is odd}. \
		\end{cases}
	\end{equation*}
	Then we see by \eqref{eq.2xm} that
	$\langle \nu_s, \alpha_j^\vee \rangle <0$,
	$\langle \nu_{s}, \alpha_{j'}^\vee \rangle >0$.
	Moreover, we see that
	$\langle \nu_{s+1}, \alpha_j^\vee \rangle >0$,
	$\langle \nu_{s+1}, \alpha_{j'}^\vee \rangle <0$
	because $m-n \in 2\Z+1$.

First, let us show that $\e_{j}\pi\in \B_r(\lambda)$ if $\e_{j}\pi \neq \mathbf{0}$.
	Take  $t_1$ and $t_0$ as (\ref{et_1}) and  (\ref{et_0}), respectively
	(with $i$ replaced by $j$).
	Since $\langle \nu_s, \alpha_j^\vee \rangle <0$
	and $\langle \nu_{s+1}, \alpha_j^\vee \rangle >0$ as seen above,
	the function $H_j^\pi(t)$ does not attain its minimum value at $t=\sigma_{s-1}, \sigma_{s+1}$
	(see the left figure in Figure \ref{fig.3}).
	Thus we obtain $t_1 \neq \sigma_{s-1}, \sigma_{s+1}$.
If $t_1<\sigma_{s-1}$, then the assertion is obvious by
the definition of $\e_j$ and Remark \ref{rem.o}.
		If $t_1> \sigma_{s+1}$, then
		$H_{j}^\pi(\sigma_{s}) >  H_{j}^\pi(t_1)= m_{j}^\pi \in \Z$ by the definition of $t_1$.
		Notice that $H_{j}^\pi(\sigma_{s}) \in \Z$ by \eqref{int},
		and hence $H_{j}^\pi(\sigma_{s}) \geq  H_{j}^\pi(t_1)+1= m_{j}^\pi +1$.
		Because $H_{j}^\pi(\sigma_{s+1}) >  H_{j}^\pi(\sigma_s) \geq m_{j}^\pi +1$,
		we see that
		$\sigma_{s+1}<t_0$.
		Therefore, $\e_{j}\pi\in \B_r(\lambda)$
		by the definition of $\e_{j}$ and Remark \ref{rem.o}.
	Assume that  $t_1= \sigma_s$.
	If $s=1$, i.e., $\sigma_{s-1}=0$, then it is obvious that $t_0\geq \sigma_{s-1}$.
	If $s>1 $, then we see that $H^\pi_j(\sigma_{s-2})>H^\pi_j(t_1) \in \Z$  by the definition of $t_1$.
	Notice that $H^\pi_j(\sigma_{s-2})\in \Z$ by \eqref{int}
	and hence $H_{j}^\pi(\sigma_{s-2}) \geq  H_{j}^\pi(t_1)+1= m_{j}^\pi +1$.
	Since $H_{j}^\pi(\sigma_{s-1}) >  H_{j}^\pi(\sigma_{s-2})\geq  m_{j}^\pi +1$
	by $\langle \nu_{s-1}, \alpha_j^\vee \rangle >0$ (see \eqref{eq.2xm} and Remark \ref{rem.o}),
	we see that
	$\sigma_{s-1}< t_0$
	by the definition of $t_0$.
	Then we deduce by the definition of $\e_j$ that
	$\e_j\pi$ is of the form
	\begin{equation*}
		\e_j\pi =
		\begin{cases}
			(r_j\nu_{s}, \nu_{s+1}, \ldots,  \nu_u;
			\sigma_0,   \ldots,  \sigma_u)
			& \text{if }  t_0=\sigma_{s-1},\\
			(\nu_1, \ldots,  \nu_{s}, r_j\nu_{s}, \nu_{s+1}, \ldots, \nu_u;
			\sigma_0,  \ldots, \sigma_{s-1}, t_0, \sigma_{s+1}, \ldots  \sigma_u)
			& \text{if } t_0 > \sigma_{s-1}.
		\end{cases}
	\end{equation*}
	Since $r_j\nu_{s}=r_j x_{m}\lambda =x_{m-1}\lambda$ and $(m-1)-n =2r$, we obtain
	$\e_j\pi\in \B_r(\lambda)$.

Next, let us show that $\e_{j'}\pi\in \B_r(\lambda)$ if $\e_{j'}\pi \neq \mathbf{0}$.
		Take  $t_1$ and $t_0$ as (\ref{et_1}) and (\ref{et_0}), respectively
		(with $i$ replaced by $j'$).
		Since $\langle \nu_s, \alpha_{j'}^\vee \rangle >0$
		and $\langle \nu_{s+1}, \alpha_{j'}^\vee \rangle <0$ as seen above,
		the function  $H_{j'}^\pi(t)$ does not attain its minimum value at $t=\sigma_{s}$
			(see the right figure in Figure \ref{fig.3}).
		Thus
		we see that $t_1\neq \sigma_{s}$.
If $t_1 \leq \sigma_{s-1}$, the assertion is obvious by
the definition of $\e_{j'}$ and Remark \ref{rem.o}.
	If $\sigma_{s+1}\leq t_1$, then
			$H_{j'}^\pi(\sigma_{s-1}) >  H_{j'}^\pi(t_1)= m_{j'}^\pi \in \Z$ by the definition of $t_1$.
			Notice that $H_{j'}^\pi(\sigma_{s-1}) \in \Z$ by \eqref{int},
			and hence $H_{j'}^\pi(\sigma_{s-1}) \geq  H_{j'}^\pi(t_1)+1= m_{j'}^\pi +1$.
			Because $H_{j'}^\pi(\sigma_{s}) >  H_{j'}^\pi(\sigma_{s-1}) \geq m_{j'}^\pi +1$,
			we see that
			$\sigma_{s} < t_0$.
Then
	$\e_{j'}\pi\in \B_r(\lambda)$
	by the definition of $\e_{j'}$ and Remark \ref{rem.o}.
	\begin{figure}
{\unitlength 0.1in%
\begin{picture}(60.5500,18.6500)(1.4000,-19.9500)%
\put(5.9500,-1.9500){\makebox(0,0){$H_j^\pi(t)$}}%
\put(37.9500,-1.9500){\makebox(0,0){$H_{j'}^\pi(t)$}}%
\put(17.9500,-16.9500){\makebox(0,0){$t=\sigma_{s}$}}%
\put(11.9500,-10.8500){\makebox(0,0){$t=\sigma_{s-1}$}}%
\put(25.9500,-4.8500){\makebox(0,0){$t=\sigma_{s+1}$}}%
\put(57.9500,-18.0500){\makebox(0,0){$t=\sigma_{s+1}$}}%
\put(49.9500,-4.8500){\makebox(0,0){$t=\sigma_{s}$}}%
\put(43.9500,-12.8500){\makebox(0,0){$t=\sigma_{s-1}$}}%
%
\special{pn 8}%
\special{pa 1195 1195}%
\special{pa 1795 1595}%
\special{fp}%
\special{pa 1795 1595}%
\special{pa 2595 595}%
\special{fp}%
\special{pa 4395 1195}%
\special{pa 4995 595}%
\special{fp}%
\special{pa 4995 595}%
\special{pa 5795 1695}%
\special{fp}%
%
\special{sh 0.300}%
\special{ia 1795 1595 10 10 0.0000000 6.2831853}%
\special{pn 8}%
\special{ar 1795 1595 10 10 0.0000000 6.2831853}%
%
\special{sh 0.300}%
\special{ia 4995 595 10 10 0.0000000 6.2831853}%
\special{pn 8}%
\special{ar 4995 595 10 10 0.0000000 6.2831853}%
%
\special{pn 8}%
\special{pa 5795 1695}%
\special{pa 5995 1495}%
\special{fp}%
\special{pa 4395 1195}%
\special{pa 4195 995}%
\special{fp}%
\special{pa 2595 595}%
\special{pa 2795 795}%
\special{fp}%
\special{pa 1195 1195}%
\special{pa 1095 1395}%
\special{fp}%
%
\special{pn 8}%
\special{pa 1095 1395}%
\special{pa 995 1595}%
\special{dt 0.045}%
\special{pa 2795 795}%
\special{pa 2995 995}%
\special{dt 0.045}%
\special{pa 4195 995}%
\special{pa 3995 795}%
\special{dt 0.045}%
\special{pa 5995 1495}%
\special{pa 6195 1295}%
\special{dt 0.045}%
\put(13.4500,-14.9500){\makebox(0,0){\tiny{$\langle \nu_{s}, \alpha_j^\vee \rangle$}}}%
\put(24.2500,-11.9500){\makebox(0,0){\tiny{$\langle \nu_{s+1}, \alpha_j^\vee \rangle$}}}%
\put(56.2500,-9.9500){\makebox(0,0){\tiny{$\langle \nu_{s+1}, \alpha_{j'}^\vee \rangle$}}}%
\put(45.2500,-7.9500){\makebox(0,0){\tiny{$\langle \nu_{s}, \alpha_{j'}^\vee \rangle$}}}%
%
\special{pn 8}%
\special{ar 1195 1195 10 10 0.0000000 6.2831853}%
%
\special{pn 8}%
\special{ar 2595 595 10 10 0.0000000 6.2831853}%
%
\special{pn 8}%
\special{ar 4395 1195 10 10 0.0000000 6.2831853}%
%
\special{pn 8}%
\special{ar 5795 1695 10 10 0.0000000 6.2831853}%
%
\special{pn 13}%
\special{pa 3795 1995}%
\special{pa 3795 295}%
\special{fp}%
\special{sh 1}%
\special{pa 3795 295}%
\special{pa 3775 362}%
\special{pa 3795 348}%
\special{pa 3815 362}%
\special{pa 3795 295}%
\special{fp}%
\special{pa 595 1995}%
\special{pa 595 295}%
\special{fp}%
\special{sh 1}%
\special{pa 595 295}%
\special{pa 575 362}%
\special{pa 595 348}%
\special{pa 615 362}%
\special{pa 595 295}%
\special{fp}%
\end{picture}}%
			\caption{\, }
			\label{fig.3}
	\end{figure}

	Similarly, we can show (in Case 2) that
		if $\f_i \pi \neq \mathbf{0}$ for $i \in I$, then $\f_i \pi \in \B_r(\lambda)$.

	$\bf{Case \ 3}$: Let $\pi \in \B_0(\lambda)$.
	Suppose, for a contradiction, that there exist
	$\pi \in \B_0(\lambda) $  and $\tilde{g} \in \{ \e_i, \f_i \mid i\in I \}$
	such that
	$\tilde{g} \pi \neq\mathbf{0}$, and $\tilde{g} \pi \notin \B_0(\lambda)$;
	we see from Remark \ref{rem.o} that
	$\pi_1:= \tilde{g} \pi \in \B_r(\lambda)$ for some   $r \geq 1$.
	Define $\tilde{h} $ by
	\begin{equation*}
		\tilde{h}:=
		\begin{cases}
			\f_i & \text{if} \ \tilde{g}= \e_i, \\
			\e_i & \text{if} \ \tilde{g}= \f_i. \\
		\end{cases}
	\end{equation*}
Then we have $\pi = \tilde{h} \pi_1$.
Because we have shown in Cases 1 and 2 that
Theorem \ref{thm.Om} is true for $r\geq 1$,
we obtain $\pi \in \B_r(\lambda)$,
which contradicts the assumption that
$\pi \in \B_0(\lambda)$.

Thus we have proved Theorem \ref{thm.Om}.
\end{proof}

%
%
%
%
\section*{Acknowledgments.}
The author would like to thank Daisuke Sagaki, who is his supervisor, for his
kind support and advice.
Also, he is grateful to the referee for valuable comments and suggestions.

\end{document}